\documentclass[11pt]{article}
\usepackage[top=2.5cm, bottom=2.5cm, left=2.5cm, right=2.5cm] {geometry}
\usepackage{lineno}
\usepackage{comment}
\usepackage{amsmath,amssymb,theorem,color,amsfonts,mathrsfs}
\numberwithin{equation}{section} 
\usepackage{amssymb}
\usepackage{color}

\usepackage{amsmath}
\usepackage{graphicx}
\usepackage{amsfonts}%
\newcommand{\R}{\ensuremath{\mathbb{R}}}

\newcommand{\N}{\ensuremath{\mathbb{N}}}

\newcommand{\cA}{\mathcal{A}}
\newcommand{\cC}{\mathscr{C}}

\newcommand{\bP}{\mathbb{P}}

\newcommand{\mP}{\mathbb{P}}
\newcommand{\Z}{\mathbb{Z}}
\newcommand{\X}{\mathbb{X}}

\newcommand{\E}{\mathbb{E}}

\newcommand{\cD}{\mathcal{D}}

\newcommand{\ta}{\mathfrak{a}}
\newcommand{\cL}{\mathcal{L}}  

\newcommand{\bx}{\mathbf x}
\newcommand{\bX}{\mathbf X}

\newcommand{\cB}{\mathcal{B}}

\newcommand{\tz}{\bar z}

\newcommand{\tpsi}{\bar{\psi}}
\newcommand{\teta}{\bar{\eta}}

\newcommand{\tp}{p{\rm -var}}
\newcommand{\tq}{q{\rm -var}}

{\Large }

\newcommand{\ltn}{\ensuremath{\left| \! \left| \! \left|}}
\newcommand{\rtn}{\ensuremath{\right| \! \right| \! \right|}}

\newtheorem{theorem}{Theorem}
{ \theorembodyfont{\normalfont} 
	\newtheorem{example}[theorem]{Example}
	\newtheorem{remark}[theorem]{Remark}
}

\newtheorem{lemma}[theorem]{Lemma}

\newtheorem{proposition}[theorem]{Proposition}

\newcounter{enumctr}

\makeatletter
\def\and{%
\end{tabular}%
\begin{tabular}[t]{c}}%
\def\@fnsymbol#1{\ensuremath{\ifcase#1\or a\or b\or c\or
		d\or e\or f\or g\or h\or i\else\@ctrerr\fi}}
\makeatother

\begin{document}
	
	\title{Strong Lyapunov functions for rough systems}
	\author{Luu Hoang Duc\thanks{Department of Mathematics, University of Klagenfurt, Austria \& Max Planck Institute for Mathematics in the Sciences, Leipzig, Germany \& Institute of Mathematics, Vietnam Academy of Science and Technology, Vietnam. {\tt \small duc.luu@aau.at, duc.luu@mis.mpg.de, lhduc@math.ac.vn}}, $\;$ J\"urgen Jost \thanks{
			J\"urgen Jost is with the Max Planck Institute for Mathematics in the Sciences, Leipzig, Germany \& ScaDS.AI, Leipzig University, Germany \& Max Planck Institute for Human Cognitive and  Brain Sciences, Leipzig, Germany \&  Santa Fe Institute, Santa Fe, USA, {\tt \small jost@mis.mpg.de}.} }
	\date{}
	\maketitle
	
	\begin{abstract}
     We extend the Lyapunov function technique, a fundamental tool for investigating asymptotic stability and existence of attractors for ordinary differential equations, by introducing the notion of a {\it strong Lyapunov function} for an autonomous drift under stochastic perturbation driven by general H\"older-continuous multiplicative noise, not necessarily Brownian.
        The mathematical setting within which our method proceeds consists of rough path calculus and the framework of random dynamical systems. We conclude that if such a function exists for the drift then the perturbed system admits a global random pullback attractor that is upper semi-continuous w.r.t. the noise intensity coefficient and the dyadic approximation of the noise. Moreover, in case the drift is globally Lipschitz continuous, then there exists a numerical attractor for the discretization which is upper semi-continuous w.r.t. the noise intensity and converges to the continuous attractor as the step size tends to zero. Several applications, including dissipative systems, the pendulum, the FitzHugh-Nagumo neuro-system and the Lorenz system, demonstrate the power of our approach. We also prove that strong Lyapunov functions can be approximated in practice by Lyapunov neural networks.
	\end{abstract}
	
	
	\section{Introduction}\label{sec:intro}
    A standard technique for showing asymptotic stability or the existence of an attractor for the autonomous differential equation
	\begin{equation}\label{eq:abstract}
		\dot{y}_t = f(y_t), \quad y \in \R^d  
	\end{equation}
    is the construction of a Lyapunov function. For instance, this method can be applied to many nonlinear systems such as the FitzHugh-Nagumo system, a stylized model of neuron dynamics widely used in computational neuroscience, or the famous Lorenz system. Namely, we can prove in such applications the existence of a so-called \emph{Lyapunov} function $V \in \cC^1(\R^d,\R_+)$ and constants $d_1,d_2$ with $d_2 >0$ such that its gradient $\nabla V$ along the vector field $f \in C(\R^d,\R^d)$ satisfies
		\begin{equation}\label{classicgradientnegativity}
			\langle \nabla V(y), f(y)\rangle \leq d_1 - d_2 V(y),\quad \forall y \in \R^d.
		\end{equation}
        Then  combining \eqref{eq:abstract} and \eqref{classicgradientnegativity} yields 
        \begin{equation}\label{eq:Lya}
            \frac{d}{dt}V(y_t) = \langle \nabla V(y_t), f(y_t)\rangle \leq d_1 - d_2 V(y_t)
        \end{equation}
        for a solution of \eqref{eq:abstract}, allowing us to deduce asymptotic exponential stability and the existence of an attractor. 
        
        Remarkably, this technique still works when \eqref{eq:abstract} is perturbed by white noise \cite{khasminskii}, that is, when we consider the stochastic differential equation
        \begin{equation}\label{eq:white}
		dy_t = f(y_t) dt + g(y_t)dB_t,
	\end{equation}
    where $B_t$ is a multi-dimensional standard Brownian motion. The reason is that we can apply the It\^o formula for a function $V \in \cC^2(\R^d,\R)$ for stochastic equation \eqref{eq:white} so that 
    \begin{equation}\label{stoch:Ito}
        d V(y_t) = \Big[\langle \nabla V(y_t), f(y_t)\rangle + \frac{1}{2} {\rm tr} \Big( g(y_t)g(y_t)^{\rm T}D^2V(y_t)\Big) \Big] dt+ \langle \nabla V(y_t), g(y_t)dB_t\rangle.
    \end{equation}
 Since the stochastic term disappears after taking the expectation in \eqref{stoch:Ito}, by imposing the condition 
 \begin{equation}\label{stoch:Lya}
     \cL V(y) = \langle \nabla V(y), f(y)\rangle + \frac{1}{2} {\rm tr} \Big( g(y)g(y)^{\rm T}D^2V(y)\Big)\leq d_1 - d_2 V(y),\quad \forall y \in \R^d
 \end{equation}
  we obtain from \eqref{stoch:Ito} and \eqref{stoch:Lya} that
  \begin{equation}\label{stoch:expect}
    \frac{d}{dt} \E V(y_t) \leq  d_1 - d_2 \E V(y_t).
    \end{equation} 
  Inequality \eqref{stoch:expect} again yields asymptotic exponential stability and the existence of an attractor in the random setting (mainly in the mean square sense \cite{kloedenlorenz}). 
  
  Unfortunately, the expectation operator no longer works for a more general stochastic system
	\begin{equation}\label{eq:stochFHN}
		dy_t = f(y_t) dt + g(y_t) dX_t,
	\end{equation}	
   where $X$ is a multi-dimensional stochastic noise that is not a semi-martingale or a Markov process (such as a fractional Brownian motion \cite{mandelbrot}). The simple reason is that, even if an It\^o type formula as \eqref{stoch:Ito} exists, in general taking the expectation does not eliminate the stochastic term in \eqref{stoch:Ito}. 
  
  Still, we are inspired by the classical idea that condition \eqref{classicgradientnegativity} depends only on the drift part, while condition \eqref{stoch:Lya} depends also on the diffusion part and might be challenging to check due to the uncertainty coming from stochastic perturbations. This is often observed in the data-driven approach (see e.g. recent work \cite{duchongcong24}) in which we only have information on the drift $f$ and very little information on the diffusion $g$. On the other hand, we observe from \eqref{stoch:Lya} that, in the special case $\|g\|_\infty, \|D^2V\|_\infty <\infty$, the second term ${\rm tr} \Big( g(y)g(y)^{\rm T}D^2V(y)\Big)$ in $\cL V(y)$ in \eqref{stoch:Lya} is bounded, thus condition \eqref{stoch:Lya} can be deduced from the classical condition \eqref{classicgradientnegativity}.
The main contribution of this paper is, therefore, to show that the Lyapunov function technique can still be applied to stochastic system \eqref{eq:stochFHN} if the diffusion part $g$ is bounded. For that purpose, we strengthen condition \eqref{classicgradientnegativity} and require that it still holds when $f$ and its argument are perturbed in a controlled way. 
     
     Behind this is the Doss-Sussmann technique \cite{Sus78} that converts the stochastic differential equation \eqref{eq:stochFHN} into a random one. For that purpose, we follow a pathwise approach to solve stochastic equation \eqref{eq:stochFHN} via Lyons' rough path theory (see \cite{lyons98}, \cite{friz} and Appendix \ref{roughpath}). Namely one attempts to solve the controlled differential equation
\begin{equation}\label{fSDE0}
	dy_t = f(y_t)dt + g(y_t)d x_t,
\end{equation}
where the driving path $x \in \R^m$ is a realization of $X$ in the space $C^{\nu}(\R, \R^d)$ of continuous paths with finite $\nu$-H\"older norm on any finite time interval, such that $x$ can be lifted to a rough path $\omega = (1,\bx) = (1,x,\X)$. The solution of \eqref{fSDE0} is often understood in the sense of either Lyons-Davie \cite{lyons98} or of Friz-Victoir \cite{friz}, which does not need to specify rough integrals. Alternatively, we can also interpret equation \eqref{fSDE0}  in the integral form for the rough path $y$ {\it controlled} by $x$ in the sense of Gubinelli \cite{gubinelli}. 

To apply rough path calculus in the simplest form, the following assumptions (which suffice for our applications) are often imposed for coefficient functions.
		
(${\textbf H}_f$) $f:\R^d \to \R^d$ is locally Lipschitz continuous.\\

(${\textbf H}_g$) $g$ is in $C^3_b(\R^d,\R^{d\times m})$ where we define
\begin{equation}\label{gcond.new}
	C_g = \|g\|_{C^3_b}:= \max \Big\{ \|g\|_\infty,\|Dg\|_\infty,\|D^2g\|_\infty,\|D^3g\|_\infty\Big\}.
\end{equation}

	(${\textbf H}_X$) For a given $\nu \in (\frac{1}{3},\frac{1}{2}]$, $\R^m \ni X_t(\omega)$ is a stochastic process with stationary increments, of which almost all realizations $x$ belong to the space $C^{\nu}(\R, \R^m)$ of $\nu-$H\"older continuous paths, such that $x$ is truly rough and can be lifted into a rough path lift $\bx = (x,\X)$ of a stochastic process $(X_\cdot(\omega),\X_{\cdot,\cdot}(\omega))$ with stationary increments, and the estimate
	\begin{equation}\label{conditionx}
		\E \Big(\|X_{s,t} \|^p +\|\X_{s,t}\|^{q}\Big)\leq C_{T,\nu} |t-s|^{p \nu },\forall s,t \in [0,T]
	\end{equation}
	holds for any $[0,T]$, with $p\nu \geq 1, q = \frac{p}{2}$ and some constant $C_{T,\nu}$. (Examples of such processes include multi-dimensional fractional Brownian motions \cite{mandelbrot}).
		
	Following \cite{duc21},\cite{duckloeden} by using rough path integrals, we solve the rough differential equation \eqref{fSDE0} using the Doss-Sussmann technique \cite{Sus78}. Recall from \cite{duc20}, \cite[Proposition 2.1]{duc21}  that the solution $\phi_{\cdot,a}(\bx,\phi_a)$ of the {\it pure} rough differential equation 
	\begin{equation}\label{pure}
		d \phi_t = g(\phi_t)dx_t,\quad t \in [a,b], \phi_a \in \R^d 
	\end{equation}
	is $C^1$ w.r.t. the initial value $\phi_a$. Moreover, for any interval $[a,b]$ such that 
    \[
    \lambda:= 16 C_p C_g \ltn \bx \rtn_{\tp,[a,b]} \leq 1, 
    \]
    where $C_p$ is a constant from the sewing lemma \cite{congduchong23}, the solution $\phi$ of \eqref{pure} satisfies the following estimates
	\begin{equation}\label{solest}
			\ltn \phi_{\cdot,a}(\bx,\phi_a) \rtn_{\tp,[a,b]}, \Big\| \Big[\frac{\partial \phi }{\partial \phi_a}(\cdot,a,\bx,\phi_a)\Big]^{-1} - Id\Big\|_{\infty,[a,b]}\leq 16 C_p C_g \ltn \bx \rtn_{\tp,[a,b]}.
	\end{equation}	
     The idea is to solve equation \eqref{fSDE0} in each small interval $[\tau_k,\tau_{k+1}]$ between two consecutive stopping times constructed on $[0,T]$, and then concatenate to obtain the conclusion on the whole interval. 
    The Doss-Sussmann technique used in \cite[Theorem 3.7]{duc20} and \cite{duckloeden} ensures 
	that, by a transformation $y_t = \phi_{t,\tau_k}(\bx,z_t)$ there is an one-to-one correspondence between a solution $y_t$ of \eqref{fSDE0} on a certain interval $[\tau_k,\tau_{k+1}]$ 
	and a  solution $z_t$ of the associated ordinary differential equation
	\begin{equation}\label{ascoODE}
		\dot{z}_t = \Big[\frac{\partial \phi }{\partial z}(t,\tau_k,\bx,z_t)\Big]^{-1} f( \phi_{t,\tau_k}(\bx,z_t)) =(Id + \psi_t) f(z_t+\eta_t),\quad t \in [\tau_k,\tau_{k+1}],\ z_{\tau_k} = y_{\tau_k},
	\end{equation}
	where we introduce the notations 
	\[
	\R^d \ni \eta_t := y_t - z_t=\phi_{t,\tau_k}(\bx,z_t)-z_t, \quad \R^{d\times d} \ni \psi_t:=\Big[\frac{\partial \phi }{\partial z}(t,\tau_k,\bx,z_t)\Big]^{-1} - Id,\quad \forall t\in [\tau_k,\tau_{k+1}].
	\]
    By constructing for any fixed $\lambda \in (0,1)$ the sequence of stopping times $\{\tau_k(\frac{\lambda}{16C_pC_g},\bx,[0,T])\}_{i \in \N}$  
	\begin{equation}\label{greedytime}
		\tau_0 = 0,\quad \tau_{k+1}:= \inf\Big\{t>\tau_k:  \ltn \bx \rtn_{\tp, [\tau_k,t]} = \frac{\lambda}{16C_pC_g} \Big\}\wedge T,
	\end{equation}
	  we can apply \eqref{solest} to consider the ODE \eqref{ascoODE} as a nonautonomous perturbation of the original ODE \eqref{eq:abstract} with
	\begin{equation}\label{HK}
		\|\eta_t\|,	\|\psi_t\| \leq \lambda,\quad \forall t \in [\tau_k,\tau_{k+1}].
	\end{equation}
As such, system \eqref{fSDE0} is proved in \cite{duc21}, \cite{duckloeden} to admit a unique continuous pathwise solution, provided that $f$ is locally Lipschitz continuous and of one-sided growth which also satisfies the linear growth in the perpendicular direction
	\begin{equation}\label{lineargrowth}
		\Big\| f(y) - \frac{\langle f(y), y \rangle }{\|y\|^2}y\Big\|\leq C_f \Big(1+\|y\|\Big),\quad \forall y \ne 0
	\end{equation}
	for a certain constant $C_f$. However, condition \eqref{lineargrowth} is difficult to apply in practice, as seen in Fitzhugh-Nagumo or Lorenz systems. This poses the challenge -- which we take up in this paper -- to find an alternative approach. 
    
	To achieve that, the Doss-Sussmann transformation motivates us to construct a Lyapunov function that works for the transformed system \eqref{ascoODE} regardless of how stochastically $\psi_t$ and $\eta_t$ can vary. In other words, the gradient of the Lyapunov function along the transformed vector field \eqref{ascoODE} should be controlled autonomously w.r.t. the perturbations $\psi$ and $\eta$ satisfying \eqref{HK}. For this purpose, a stronger condition than \eqref{classicgradientnegativity} should be imposed in the sense that it should be free of the diffusion part and also the noise. We propose a new notion of {\it strong Lyapunov functions} which requires, apart from controlling the growth of $V$ and its gradient, the existence of a small constant $\lambda \in (0,1)$ and constants $C_\lambda, \delta$ with $\delta  >0$, such that
		\begin{equation}\label{gradientnegativity1}
			\sup \limits_{\substack{\psi \in \R^{d\times d},\|\psi\|\leq \lambda \\\eta \in \R^d,\|\eta\| \leq \lambda}} \langle \nabla V(z), (I + \psi) f(z+\eta)\rangle \leq C_\lambda - \delta V(z),\quad \forall z \in \R^d.
		\end{equation}  
    Obviously, condition \eqref{gradientnegativity1} is stronger than the classical condition \eqref{classicgradientnegativity} (which corresponds to $\lambda =0, C_\lambda = d_1, \delta = d_2$).    
    In applications, it turns out that the strong condition \eqref{gradientnegativity1}  holds for  dissipative drifts that satisfy condition \eqref{lineargrowth}, as seen in \cite[Example 1.2]{duc21}. The condition is also a direct consequence of the classical one \eqref{classicgradientnegativity} in case the drift $f$ is globally Lipschitz continuous (see Theorem \ref{Lipschitzf}) such as the pendulum. Also, when dealing with systems with additive noise, the role of $\psi$ in \eqref{gradientnegativity1} and of the rough path lift $\bx$ can be dropped (see Remark \ref{remadditive}), and we obtain \eqref{gradientnegativity1} from the classical condition \eqref{classicgradientnegativity}, see Theorem \ref{addLyaV} which is applicable for the Lorenz system in Examples \ref{exlorenz} and \ref{lorenzadditive}. 
    For complicated nonlinear drifts $f$ like the FitzHugh-Nagumo neuro system \cite{fitzhugh61}, we prove in Theorem \ref{roughFHNthm} that such an explicit strong Lyapunov function can be found. 
        
    If such a function exists, the stochastic system \eqref{eq:stochFHN} with its pathwise interpretation \eqref{fSDE0} generates a continuous random dynamical system $\varphi$ \cite{arnold}, which admits a global pullback attractor $\cA$ (see Theorem \ref{existenceuniqueness}, Theorem \ref{attractorgbounded2}). Moreover, Theorem \ref{attractorgboundedsemicont} and Theorem \ref{dyadicconvergence} show that the pullback attractor is upper semi-continuous w.r.t. $C_g$ and can be approximated in the pathwise sense using a dyadic approximation. If the drift $f$ is globally Lipschitz continuous, then there exists the numerical random pullback attractor $\cA^\Delta$ for the rough Euler scheme of \eqref{fSDE0} with a regular grid of step size $\Delta$. Moreover, the numerical attractor $\cA^\Delta$ is also upper semi-continuous w.r.t. $C_g$ almost surely and approximates well the continuous attractor $\cA$ as the step size $\Delta$ converges to zero (see Theorem \ref{attractordisc1}). In addition, the topic of local stability can also be studied with the help of strong Lyapunov functions, as proved in Theorem \ref{localstabLya} in Subsection \ref{localstabsec}.	
  
    Since the task of finding a strong Lyapunov function is quite challenging in general, we discuss in Section \ref{Lyanetsec} the approximation of strong Lyapunov functions by Lyapunov neural networks on a compact domain. We prove in Theorem \ref{accLya} that such an approximation is feasible up to an error and a probability arbitrarily close to one.  

    Our paper consistently uses an ODE approach to facilitate reading for people not familiar with rough path calculus. 
	
    \section{Strong Lyapunov functions}
    
    Recall that a function $\kappa: \R_+ \to \R_+$ is a $\mathcal{K}^\prime_\infty$ function if it is continuous, strictly increasing and $\lim \limits_{t \to \infty} \kappa(t) = \infty$. A function $\kappa \in \mathcal{K}^\prime_\infty$ is called a $\mathcal{K}_\infty$ function if, in addition, $\kappa(0)=0$. 
    A function $\kappa \in \mathcal{K}^\prime_\infty$ is called a $\mathcal{K}^{\rm tempered}_\infty$ function if it satisfies 
    \begin{equation}\label{tempered}
        \limsup \limits_{\epsilon \to 0} \limsup \limits_{t \to \infty}  \frac{1}{t} \log \kappa(e^{\epsilon t}) =0.
    \end{equation}
    We call a function  $\kappa \in \mathcal{K}^\prime_\infty$ a $\mathcal{K}^{\rm poly}_\infty$ if there exists constants $C_\kappa,\rho_\kappa >0$ such that 
    \begin{equation}\label{polyest}
        \kappa (t) \leq C_\kappa (1+t^{\rho_\kappa}),\quad \forall t \in \R_+.
    \end{equation}
    Note that if $\kappa \in \mathcal{K}^\prime_\infty$ then so is its inverse function $\kappa^{-1}$. In addition, it is easily seen that 
    \[
    \mathcal{K}^{\rm poly}_\infty \subset \mathcal{K}^{\rm tempered}_\infty. 
    \]
    Indeed, we can prove that (see the proof in Appendix).
    \begin{lemma}\label{Kfunctions}
If $\alpha, \beta \in \mathcal{K}^{\rm tempered}_\infty$ then so does $\beta \circ \alpha$. Also, if $\alpha, \beta \in \mathcal{K}^{\rm poly}_\infty$ then so does $\beta \circ \alpha$.
    \end{lemma}
    The following assumption is imposed for the strong Lyapunov functions.\\
	
	(${\textbf H}_V$) There exists for a strong Lyapunov function $V \in C^1(\R^d,\R_+)$ for the drift $f$ that satisfies 
	\begin{itemize}
		\item there exists functions $\alpha,\beta \in \mathcal{K}^{\rm tempered}_\infty$ such that $\alpha^{-1}\in \mathcal{K}^{\rm tempered}_\infty$ and
		\begin{equation}\label{Vbound} 
			\alpha(\|z\|) \leq V(z) \leq \beta(\|z\|),\quad \forall z \in \R^d;	
		\end{equation}
		\item there exists a constant $L_V > 0$ such that
		\begin{equation}\label{lipschitzV}
			\| \nabla V(z)\| \leq L_V,\quad \forall z \in \R^d;
		\end{equation}
		\item there exists a constant $\lambda \in (0,1)$ associated with parameters $C_\lambda, \delta >0$ such that
		\begin{equation}\label{gradientnegativity}
			\sup \limits_{\substack{\psi \in \R^{d\times d},\|\psi\|\leq \lambda \\\eta \in \R^d,\|\eta\| \leq \lambda}}  \langle \nabla V(z), (I + \psi) f(z+\eta)\rangle \leq C_\lambda - \delta V(z),\quad \forall z \in \R^d.
		\end{equation}
	\end{itemize}
	\begin{remark}\label{remLgamma}
    i, In general, condition \eqref{gradientnegativity} can be written in the form
    \begin{equation}\label{gradientgeneral}
        \sup \limits_{\substack{\psi \in \R^{d\times d},\|\psi\|\leq \lambda \\\eta \in \R^d,\|\eta\| \leq \lambda}}  \langle \nabla V(z), (I + \psi) f(z+\eta)\rangle \leq \gamma(V(z)),\quad \forall z \in \R^d,
    \end{equation}
    for a certain function $\gamma$ which is one-sided globally Lipschitz continuous w.r.t. the Lipschitz constant $L_\gamma$. As seen in Theorem \ref{existenceuniqueness} in Section \ref{attractorsec}, \eqref{gradientgeneral} is enough to prove the existence and uniqueness theorem. However, for the purpose of attractor theory, we mostly consider in the paper the simplest case $\gamma (u) = C_\lambda -\delta u$.
    
ii, Note that if \eqref{gradientnegativity} is satisfied for a certain $\lambda_0 \in (0,1)$, then it is also satisfied for any $0\leq\lambda < \lambda_0$ with the same $C_{\lambda_0}$ and $\delta$. Also, condition \eqref{gradientnegativity} is  motivated from the observation that 
        \begin{equation}\label{DSeta}
        \eta_t =\phi_{t,\tau_k}(\bx,z_t)-z_t = \int_{\tau_k}^t g(\phi_{s,\tau_k}(\bx,z_t))dx_s,\quad t\in [\tau_k,\tau_{k+1}]
        \end{equation}
        varies as stochastically as $x_t$ on $[\tau_k,\tau_{k+1}]$, a simple example is an additive noise $g(y)\equiv 1$ with $X = B^H$ to be a fractional Brownian motion for $H \in (\frac{1}{3},1)$. We would like to build a time independent mechanism that controls the action of the Lyapunov function along the transformed system \eqref{ascoODE} such that it must be bounded by the right hand side of \eqref{gradientnegativity} regardless of any stopping time interval $[\tau_k,\tau_{k+1}]$. Because of \eqref{HK}, the supremum in the left hand side of \eqref{gradientnegativity} should be imposed.
        
iii, In case of additive noise, i.e. $g(\cdot) \equiv \bar{g}\in \cL(\R^m,\R^d)$ is a constant matrix, it follows from \eqref{DSeta} that 
\begin{equation}\label{etaadditive}
\eta_t = \bar{g}x_{\tau_k,t},\quad \forall t\in [\tau_k,\tau_{k+1}];
\end{equation}
thus $\psi \equiv 0$ and condition \eqref{gradientnegativity} can be modified into a simpler form
		\begin{equation}\label{addgradientnegativity}
			\sup \limits_{\eta \in \R^d,\|\eta\| \leq \lambda}  \langle \nabla V(z), f(z+\eta)\rangle \leq C_\lambda - \delta V(z),\quad \forall z \in \R^d,
		\end{equation}
        for a certain constant $\lambda \in (0,1)$ associated with parameters $C_\lambda, \delta >0$.
	\end{remark}	
We will discuss below several situations in which condition \eqref{gradientnegativity} can be relaxed to \eqref{classicgradientnegativity}.	There are also systems with which one can construct explicit strong Lyapunov function for dealing with additive noise.
	
	\subsection{Globally Lipschitz continuous drifts}\label{subLipschitz}
	We first apply our new concept of Lyapunov functions to the vector field $f$ that is globally Lipschitz continuous, i.e. there exists $C_f$ such that
	\begin{equation}\label{lipschitzf}
		\|f(z_1)-f(z_2)\| \leq C_f\|z_1-z_2\|,\quad \forall z_1,z_2 \in \R^d.
	\end{equation}
	The following theorem asserts that for a globally Lipschitz continuous drift $f$, condition \eqref{gradientnegativity} can be derived from condition \eqref{classicgradientnegativity}.
	\begin{theorem}\label{Lipschitzf}
	i,	Assume that $f$ satisfies \eqref{lipschitzf} and there exists a classical Lyapunov function $V$ satisfying \eqref{classicgradientnegativity}, \eqref{lipschitzV} and \eqref{Vbound} in the form 
        \begin{equation}\label{Vbounds}
        \exists\  \alpha,\beta >0, C\in \R:\quad \alpha \|z\| + C\leq V(z) \leq \beta (1+\|z\|), \quad \forall z\in \R^d.        
        \end{equation}
     Then $V$ is a strong Lyapunov function which satisfies condition \eqref{gradientnegativity} by choosing $\lambda \in (0,1)$ sufficiently small and 
     \begin{equation}\label{Clambdadelta}
         C_\lambda:=  d_1 + L_V \lambda (2C_f \lambda + \|f(0)\|-\frac{C}{\alpha}C_f),\quad \delta:= d_2-\frac{1}{\alpha}L_V C_f\lambda.
     \end{equation}
     ii, Assume the Lipschitz condition \eqref{lipschitzf} for $f$ and there exists a classical Lyapunov function satisfying \eqref{classicgradientnegativity}, \eqref{lipschitzV}, \eqref{Vbound} and further 
        \begin{equation}\label{Vbounds2}
        \exists\  K >0, C\in \R:\quad \frac{1}{K}\|f(z)\| + C\leq V(z), \quad \forall z\in \R^d.        
        \end{equation}
     Then $V$ is a strong Lyapunov function which satisfies condition \eqref{gradientnegativity} by choosing $\lambda$ sufficiently small and 
     \begin{equation}\label{Clambdadelta2}
         C_\lambda:=  d_1 + L_V \lambda (K L_V\lambda -KC+ C_f),\quad \delta:= d_2-KL_V \lambda.
     \end{equation}
	\end{theorem}
      
    	\begin{proof}
		i, Obviously, the lower and upper bound functions in \eqref{Vbounds} are in class $\mathcal{K}_\infty^{\rm poly}$. Observe that it follows from \eqref{lipschitzV} and \eqref{lipschitzf} that
         \allowdisplaybreaks
		\begin{eqnarray}\label{Lipschitzfestimate}
			&&\langle \nabla V(z), (I + \psi)f(z+\eta) \rangle \notag\\
            &=& \langle \nabla V(z), f(z) \rangle + \langle \nabla V(z), \psi f(z+\eta) \rangle + \langle \nabla V(z), f(z+\eta) - f(z) \rangle \notag\\
			&\leq& d_1 - d_2 V(z) + \|\nabla V(z)\| \Big(\|\psi\| \|f(z+\eta)\| + \|f(z+\eta)-f(z)\|\Big)  \notag\\
			&\leq& d_1 - d_2 V(z) + L_V \lambda \Big[ C_f(\|z\| + \|\eta\|) + \|f(0)\| + C_f \|\eta\| \Big] \notag\\
			&\leq& d_1 + L_V \lambda (2C_f \lambda + \|f(0)\|) + L_V C_f\lambda  \|z\| - d_2 V(z) \notag\\
			&\leq& d_1 + L_V\lambda (2C_f \lambda + \|f(0)\|-\frac{C}{\alpha}C_f) + \Big(\frac{1}{\alpha}L_V C_f\lambda - d_2\Big) V(z).
		\end{eqnarray}
		By choosing $\lambda$ small enough such that $\delta, C_\lambda$ in \eqref{Clambdadelta} are positive, we obtain \eqref{gradientnegativity}.\\
	
   ii, Similarly, \eqref{Lipschitzfestimate} is re-estimated as follows
     \allowdisplaybreaks
        \begin{eqnarray*}
			&&\langle \nabla V(z), (I + \psi)f(z+\eta) \rangle \notag\\
            &=& \langle \nabla V(z), f(z) \rangle + \langle \nabla V(z), \psi f(z+\eta) \rangle + \langle \nabla V(z), f(z+\eta) - f(z) \rangle \notag\\
			&\leq& d_1 - d_2 V(z) + \|\nabla V(z)\| \Big(\|\psi\| \|f(z+\eta)\| + \|f(z+\eta)-f(z)\|\Big)  \notag\\
			&\leq& d_1 - d_2 V(z) + L_V \Big[-K\|\psi\|C+\|\psi\|K V(z+\eta) + C_f \|\eta\| \Big] \notag\\
			&\leq& d_1 - KCL_V\lambda- d_2 V(z) + L_V \Big[K\|\psi\| \Big(V(z) + L_V\|\eta\|\Big) + C_f \|\eta\|\Big] \notag\\
			&\leq& d_1 + L_V\lambda (-KC+K L_V\lambda + C_f) + \Big(KL_V\lambda - d_2\Big) V(z).
		\end{eqnarray*}
        By choosing $\lambda$ small enough such that $\delta,C_\lambda$ in \eqref{Clambdadelta2} are positive, we deduce \eqref{gradientnegativity}.\\
    \end{proof}
	\begin{example}\label{penex}[Pendulum system]
		An indirect application is the pendulum system in 2 dimension $z = (v,w) \in \R^2$ given by 
		\begin{equation}\label{pendulum}
			\begin{cases}
				\dot{v}_t &= w_t\\
				\dot{w}_t &= - \sigma^2 \sin v_t - 2 \mu w_t
			\end{cases}
		\end{equation}
		Then $f$ is globally Lipschitz continuous with a certain constant $C_f$. Introduce function 
        \[
        V(z) := \sqrt{1+ \frac{1}{2} w^2 + \sigma^2 (1-\cos v)}. 
        \]
        It is easy to check that $\|\nabla V\|$  are bounded by a constant $L_V = 1+\sigma^2$ and 
    	\[
		\langle \nabla V(z),f(z) \rangle = \frac{1}{2V(z)} \Big[\sigma^2 w\sin v - w(\sigma^2 \sin v + 2 \mu w) \Big] = -\frac{\mu w^2}{V(z)} \leq 2\mu(1+\sigma^2)-2\mu V(z).
		\]
  	Note that $V$ does not satisfy condition \eqref{Vbound} and \eqref{Vbounds} globally since the first inequality in \eqref{Vbound} and in \eqref{Vbounds} does not hold on a domain with unbounded $v$. Thus one can not apply directly Theorem \ref{Lipschitzf}. However, observe that
        \begin{eqnarray*}
        \|f(z)\| &=& [w^2 + (\sigma^2 \sin v + 2\mu w)^2]^{\frac{1}{2}} \\
        &\leq& [w^2 + 2 \sigma^4 (\sin v)^2 + 8 \mu^2 w^2]^{\frac{1}{2}} \\
        &\leq& [(2+16 \mu^2) \frac{1}{2}w^2 +2 \sigma^4 + 2\sigma^6(1-\cos v)]^{\frac{1}{2}} \\
        &\leq& [(2+16\mu^2) \vee 2\sigma^4]^{\frac{1}{2}} V(z) = K V(z).    
        \end{eqnarray*}
    Hence $V$ satisfies \eqref{Vbounds2} and one can apply Theorem \ref{Lipschitzf} to obtain \eqref{gradientnegativity}. 
        \end{example}

    \subsection{Dissipative systems}
    Consider $f$ to satisfy the global Lipschitz continuity and the dissipativity condition, i.e. there exists $d_1\geq 0, d_2 > 0$ such that
		\begin{equation}\label{dissipative}
			\langle z, f(z) \rangle \leq d_1 - d_2 \|z\|^2,\quad \forall z \in \R^d.
		\end{equation}	
    When $f$ is dissipative but not global Lipschitz continuous, one needs to impose additional condition of linear growth in the perpendicular direction \eqref{lineargrowth}, see e.g. \cite[Example 1.2]{duc21} for two conventional examples
\begin{equation*}
    f(y) = \chi y - \|y\|^2 y,\quad \forall y \in \R^d; \quad \text{or} \quad  f(y) 	= \Big(\begin{array}{cc} b y_2 + y_1(a-y_1^2-y_2^2) \\ -b y_1 + y_2(a-y_1^2-y_2^2) \end{array}\Big),\ \forall y = (y_1,y_2)^{\rm T} \in \R^2.
\end{equation*} 
	Then it is easy to check that $V(z) := \sqrt{1+\|z\|^2}$ is a strong Lyapunov function that satisfies conditions \eqref{Vbounds} (due to $\|z\| \leq V(z) \leq \|z\| +1$), \eqref{lipschitzV} (with $\|\nabla V\| \leq 1$) and \eqref{classicgradientnegativity}. \\
    It is proved in \cite[Theorem 2.1, formular (39)]{duc21} that there exists $\bar{C}_\lambda,\delta>0$ such that
    \[
  \sup \limits_{\substack{\psi \in \R^{d\times d},\|\psi\|\leq \lambda \\\eta \in \R^d,\|\eta\| \leq \lambda}}  \langle z, (I + \psi) f(z+\eta)\rangle \leq \bar{C}_\lambda-\delta\|z\|^2,\quad \forall z \in \R^d.
    \]
    Hence, we obtain for our $V(z) := \sqrt{1+\|z\|^2}$ the estimate
           \allowdisplaybreaks
		\begin{eqnarray*}
	&&\sup \limits_{\substack{\psi \in \R^{d\times d},\|\psi\|\leq\lambda \\\eta \in \R^d,\|\eta\| \leq \lambda}}  \langle \nabla V(z), (I + \psi) f(z+\eta)\rangle 
    = \frac{1}{V(z)}\sup \limits_{\substack{\psi \in \R^{d\times d},\|\psi\|\leq\lambda \\\eta \in \R^d,\|\eta\| \leq \lambda}}  \langle z, (I + \psi) f(z+\eta)\rangle \\
    &\leq& \frac{\bar{C}_\lambda-\delta \|z\|^2}{V(z)}
    \leq \frac{\bar{C}_\lambda+\delta-\delta (1+\|z\|^2)}{V(z)}
    \leq \bar{C}_\lambda+\delta -\delta V(z),\quad \forall z \in \R^d,
  	\end{eqnarray*}
  	which shows that $V$ satisfies \eqref{gradientnegativity} and thus is a strong Lyapunov function.	

       \subsection{Systems with additive noises}
    In several situations, it happens that one works with perturbations of additive noises and thus needs to construct a strong Lyapunov function which satisfies condition \eqref{addgradientnegativity}. The following result gives us a criterion to test classical Lyapunov functions satisfying \eqref{classicgradientnegativity}.

    \begin{theorem}\label{addLyaV}
        Assume $V \in \cC^2$ satisfies \eqref{Vbound}, \eqref{lipschitzV} and the classical condition \eqref{classicgradientnegativity}. If in addition,
        \begin{equation}\label{deriv2V}
        \exists C_f >0: \|f(z)\|  \sup_{\eta \in \R^d,\|\eta\|\leq 1}\| D^2 V(z+\eta) \|\leq C_f(1+V(z))\quad \forall z \in \R^d,
        \end{equation}
        then $V$ is a strong Lyapunov function which satisfies \eqref{addgradientnegativity} by choosing $\lambda$ sufficiently small and 
        \begin{equation}\label{addnoiseparam}
            C_\lambda := d_1 +\lambda C_f + L_V \lambda (d_2-C_f\lambda),\quad \delta:= d_2-\lambda C_f.
        \end{equation}
    \end{theorem}
    \begin{proof}
    By Lagrange's mean value theorem, for $\|\eta\|\leq \lambda <1$, there exists $\gamma \in (0,1)$ such that 
    \[
    \nabla V(z+\eta)-\nabla V(z) =D^2 V(z+\gamma\eta) \eta. 
    \]
    Note that $|\gamma-1| \|\eta\| \leq 1$, which yields
$z+\eta-1 \leq z +\gamma \eta \leq z + \eta +1$, or equivalently $z+\gamma \eta \in B(z+\eta,1)$. A direct computation shows that 
        \begin{eqnarray*}
     \langle \nabla V(z), f(z+\eta)\rangle &\leq&  \langle \nabla V(z+\eta), f(z+\eta)\rangle  + \langle \nabla V(z)-\nabla V(z+\eta),f(z +\eta)\rangle \\
     &\leq& d_1 - d_2 V(z +\eta) + \|f(z+\eta)\|\|D^2 V(z+\gamma\eta)\| \|\eta\| \\
     &\leq& d_1 - d_2 V(z +\eta) + \lambda\|f(z+\eta)\|\sup \limits_{\eta^* \in \R^d,\|\eta^*\| \leq 1} \|D^2 V(z+\eta+\eta^*)\| \\
     &\leq& d_1 - d_2 V(z +\eta) + \lambda C_f\Big(1+V(z+\eta)\Big) \\
     &\leq& d_1 +\lambda C_f - (d_2-\lambda C_f) V(z +\eta) \\
     &\leq& d_1 +\lambda C_f +|d_2-\lambda C_f| |V(z)-V(z+\eta)| - (d_2-\lambda C_f) V(z)  \\
     &\leq& d_1 +\lambda C_f + \lambda L_V |d_2-\lambda C_f| - (d_2-\lambda C_f) V(z),
     \end{eqnarray*}
     where the last inequality is due to \eqref{lipschitzV}.
     This shows that $V$ satisfies \eqref{addgradientnegativity} with $\lambda < (\frac{d_2}{C_f} \wedge 1)$ and $C_\lambda, \delta$ satisfying \eqref{addnoiseparam}.
    \end{proof}

     \begin{example}\label{exlorenz}[Lorenz attractor]
           Consider the Lorenz system in three dimension $\xi = (x,y,z)^{\rm T}$ of the form
    \begin{equation}\label{lorenz}
    \begin{cases}
        \dot{x} &= \sigma(y-x) \\
        \dot{y} &= x(\rho-z) -y \\
        \dot{z} &= -\beta z +xy
    \end{cases}    
    \end{equation} 
where $\sigma, \rho, \beta >0$ are parameters. There are some versions of classical Lyapunov functions for system \eqref{lorenz}. We also refer to \cite{schmalfuss97} and \cite{arnoldschmalfuss} for a study of random attractor for system \eqref{lorenz} under stochastic linear Stratonovich noise, by constructing a random Lyapunov function w.r.t. a forward invariant pullback absorbing set. However such a construction is noise dependent.

To construct a strong Lyapunov function candidate for system \eqref{lorenz}, consider the function 
\begin{equation}\label{Lyalorenz}
V(\xi) = \Big[1+\frac{1}{\sigma} x^2 + y^2 + (z-\rho)^2\Big]^{\frac{1}{2}} \geq 1. 
\end{equation}
Then it is easy to check that 
\[
\Big(\frac{1}{\sqrt{1+\rho^2}} \wedge \frac{1}{\sqrt{\sigma}}\Big) \|\xi\| \leq V(\xi)  \leq (\frac{1}{\sqrt{\sigma}} \vee \sqrt{2})\|\xi\| + \sqrt{1+2\rho^2},
\]
 where the bound functions of $V(\xi)$ are in the class $\mathcal{K}_\infty^{\rm poly}$. In addition,
\begin{eqnarray*}
 && \nabla V(\xi) = \frac{1}{V(\xi)} \Big(\frac{1}{\sigma} x,y,z-\rho\Big)\Rightarrow L_V := \sup_{\xi \in \R^3}  \|\nabla V(\xi)\|  \leq (\frac{1}{\sqrt{\sigma}} \vee 1);\\
 && D^2 V(\xi) = \frac{1}{V(\xi)^3}\left(\begin{matrix}
     \frac{1}{\sigma} V - \frac{1}{\sigma^2}x^2 & -\frac{1}{\sigma}xy & -\frac{1}{\sigma} x(z-\rho)\\ -\frac{1}{\sigma}xy& V^2 - y^2 & y(z-\rho) \\ -\frac{1}{\sigma} x(z-\rho) & y(z-\rho)& V- (z-\rho)^2     
 \end{matrix} \right)\Rightarrow \|D^2 V(\xi)\| \leq \frac{\kappa_V}{V(\xi)};\quad \\
 && \|f(\xi)\| \leq \max \Big\{\sigma +1, \rho +1,\beta\Big\} \Big(|x|+|y|+|z|+|xy|+|xz|\Big) \leq \kappa_f V(\xi)^2;
\end{eqnarray*}
for certain constants $\kappa_V,\kappa_f$, i.e. $V$ satisfies conditions \eqref{Vbound}, \eqref{lipschitzV}. On the other hand, there exists $\eta^* \in \R^d, \|\eta^*\| \leq 1$ such that $\min_{\eta \in \R^3,\|\eta\|\leq 1} V(\xi+\eta) = V(\xi+\eta^*)$. Thus
\begin{eqnarray*}
&&\|f(\xi)\|\sup_{\eta \in \R^3,\|\eta\|\leq 1}\|D^2 V(\xi +\eta)\| \leq \kappa_f V(\xi)^2 \sup_{\eta \in \R^3,\|\eta\|\leq 1} \frac{\kappa_V}{V(\xi+\eta)} = \kappa_f \kappa_V V(\xi)^2 \frac{1}{V(\xi+\eta^*)}\\
&\leq& \kappa_f \kappa_V \Big[V(\xi+\eta^*) + L_V\|\eta^*\| \Big]^2 \frac{1}{V(\xi+\eta^*)}\leq \kappa_f \kappa_V \Big[V(\xi+\eta^*) + 2 L_V\|\eta^*\|+ L_V^2 \|\eta^*\|^2\Big] \\
&\leq& \kappa_f \kappa_V \Big[V(\xi) + 3 L_V\|\eta^*\|+ L_V^2 \|\eta^*\|^2\Big] 
\leq \kappa_f \kappa_V (1+3 L_V+3 L_V^2)\big[1+V(\xi)\big]. 
\end{eqnarray*}
This shows that $V$ satisfies \eqref{deriv2V} for $C_f = \kappa_V \kappa_f(1+3 L_V+3 L_V^2)$. Therefore, we conclude from Theorem \ref{addLyaV} that $V$ is a strong Lyapunov function for the Lorenz system \eqref{lorenz}.  

\end{example}
   
 \begin{remark}
     In general, it is a difficult task to find an explicit strong Lyapunov function for the Lorenz system \eqref{lorenz} that satisfies the general condition \eqref{gradientnegativity}. 
 \end{remark}
 
	\subsection{Fitzhugh-Nagumo neuro-system}\label{FHNsec}
	
	In our previous study \cite{ducjostdatmarius}, we consider the long time behaviour of the deterministic  FitzHugh-Nagumo (FHN) neuron model 
	\begin{equation}\label{eq:dFHN}
		\begin{cases}
			\dot{v}_t &= v_t-\displaystyle{\frac{v_t^3}{3}}-w_{t}+I\\
			\dot{w}_t &= \varepsilon (v_t-\mu w_t + J)
		\end{cases}
	\end{equation}
	under a perturbation of stochastic noises. 
    It is well known (see e.g. \cite{ducjostdatmarius}) that the vector field $f$ is dissipative in the sense of \eqref{dissipative}, hence there exists a global attractor for \eqref{eq:dFHN} which contains one or more equilibria. We prove in \cite{ducjostdatmarius} that under a special type of  multiplicative standard Brownian noise, the perturbed system of \eqref{eq:dFHN} admits a global random pullback attractor. The question of existence of attractor is still open when the noise is not of standard Brownian type. Note that system \eqref{eq:dFHN} does not satisfy the linear growth in the perpendicular direction \eqref{lineargrowth}, hence one can not apply directly the arguments in \cite{duc21} and \cite{duckloeden} for the standard Lyapunov function $\bar{V}(y) = \|y\|^2$ to prove the existence and uniqueness theorem. 
    
    Instead, consider a smooth function
	\begin{equation}\label{LyaV}
		V(y):=\Big(1+ v^4 + B w^2\Big)^{\frac{1}{4}},
	\end{equation}
	where $B:= \frac{6}{\varepsilon \mu}$. Observe that $V$ satisfies condition \eqref{Vbound}, where it follows from Cauchy inequality that
	\begin{equation}\label{Vest}
		\begin{split}
			V(z) &\geq (2v^2 + Bw^2)^{\frac{1}{4}} \geq (2 \wedge B)^{\frac{1}{4}} \|z\|^{\frac{1}{2}} \\
			V(z) & \leq (1+\frac{B}{2} + v^4 + \frac{B}{2}w^4)^{\frac{1}{4}} \leq  (1+\frac{B}{2})^{\frac{1}{4}} (1+\|z\|),
		\end{split}
	\end{equation}
where the functions in the right hand sides of \eqref{Vest} are in class $\mathcal{K}_\infty^{\rm poly}$. On the other hand,
	\begin{eqnarray*}
		\nabla V(z) &=& \frac{(4 v^3, 2B w)}{4\Big(1+v^4 + B w^2\Big)^{\frac{3}{4}}}, 
	\end{eqnarray*}
	where one can apply Cauchy inequality to obtain
	\begin{eqnarray*}
		4(1+v^4+Bw^2)^{\frac{3}{4}} &\geq& 4(v^4)^{\frac{3}{4}} = 4|v|^3;\\	
		4(1+v^4+Bw^2)^{\frac{3}{4}} &\geq&	4\Big(1+ \frac{Bw^2}{2} + \frac{Bw^2}{2}\Big)^{\frac{3}{4}}\geq 4\Big(\frac{3 B^{\frac{2}{3}}}{2^{\frac{2}{3}}} |w|^{\frac{4}{3}} \Big)^{\frac{3}{4}} =4 \frac{3^{\frac{3}{4}}B^{\frac{1}{2}}}{2^{\frac{1}{2}}}|w|.
	\end{eqnarray*}
	Thus, $V$ satisfies condition \eqref{lipschitzV}. Moreover we can show that $V$ also satisfies \eqref{gradientnegativity} to conclude the following result.
	\begin{theorem}\label{roughFHNthm}
		For the FHN system, the function $V$ in \eqref{LyaV} is a strong Lyapunov function.
	\end{theorem}
 
        	\begin{proof}
		With little abuse of notation, one writes 
		\begin{equation}
			\psi = \left( \begin{array}{ccc}
				a & b \\ c& d
			\end{array} \right),\quad z = \left( \begin{array}{ccc}
				v \\ w 
			\end{array} \right),\quad  \eta = \left( \begin{array}{ccc}
				\eta_1 \\ \eta_2
			\end{array} \right),
		\end{equation}
		so that 
		\begin{equation}\label{transformed} 
		(Id + \psi) f(z+\eta) = \left( \begin{array}{ccc}
				1+a & b \\ c& 1+d
			\end{array} \right) \left( \begin{array}{ccc}
				(v+\eta_1) - \frac{1}{3}(v+\eta_1)^3 - (w + \eta_2) + I \\ \varepsilon \Big[(v+\eta_1)- \mu (w+\eta_2) + J\Big]
			\end{array} \right).
		\end{equation}
		Due to \eqref{HK},
		\begin{equation}
			0 \leq |a|, |b|, |c|, |d|, |\eta_1|, |\eta_2| \leq \lambda <1. 
		\end{equation}
	A direct computation to estimate $V$ along the vector field \eqref{transformed} shows that
     \allowdisplaybreaks
		\begin{equation}\label{Vdot}
        \begin{split}
			&	4 V(z)^3 \langle \nabla V(z), (Id + \psi) f(z+\eta)\rangle \\
			= & 4\Big\{(1+a)v^3 \Big[ (v+\eta_1) - \frac{1}{3}(v+\eta_1)^3 - (w + \eta_2) + I\Big] + b \varepsilon v^3  \Big[(v+\eta_1)- \mu (w+\eta_2) + J\Big] \Big\} \\
			& + 2B \Big\{cw\Big[(v+\eta_1) - \frac{1}{3}(v+\eta_1)^3 - (w + \eta_2) + I\Big] +(1+d) \varepsilon w \Big[(v+\eta_1)- \mu (w+\eta_2) + J\Big] \Big\}\\
			\leq & 4 v^3 (v- \frac{1}{3} v^3 - w + I ) + 2B\varepsilon w (v -\mu w+ J) + \lambda Q_1(\lambda,|v|) + \lambda |w|Q_2(\lambda,|v|) + \lambda Q_3(\lambda,|w|) \\
			\leq & 4 v^3 (v- \frac{1}{3} v^3 - w + I ) + 2B\varepsilon w (v -\mu w+ J)+ \lambda Q_1(\lambda,|v|) + \lambda Q_3(\lambda,|w|) + \lambda w^2 + \lambda Q_2(\lambda,|v|)^2 
         \end{split}   
		\end{equation} 
		where $Q_1$ is a polynomial of $|v|$ with order 5, $Q_2$ is a polynomial of $|v|$ with order 3, and $Q_3$ is a polynomial of $|w|$ with order 2; all of them have positive coefficients dependent on $\lambda$. Using the Cauchy inequality, one can easily shows that
		\begin{equation}\label{cauchy}
			|w| \leq \frac{1}{2} w^2 + \frac{1}{2},\quad \text{and} \quad |v|^i \leq \frac{i}{6} v^6 + \frac{6-i}{6},\quad \forall i = 1 \ldots 5.
		\end{equation} 
		Hence, there exists a generic constant $D>1$ (independent of $\lambda$) such that 
		\[
		\lambda Q_1(\lambda,|v|) + \lambda Q_3(\lambda,|w|) 
		+ \lambda w^2 + \lambda Q_2(\lambda,|v|)^2 \leq D (\lambda v^6 + \lambda w^2+ 1).
		\]
		As a result, inequality \eqref{Vdot} can be estimated further as
         \allowdisplaybreaks
		\begin{eqnarray}\label{Vdot2}
			&&	4 V(z)^3 \langle \nabla V(z), (Id + \psi) f(z+\eta)\rangle  \notag\\
			&\leq &  (D \lambda - \frac{4}{3}) v^6 +  4 v^3 w+4v^4+4I v^3 + (D\lambda- 2B\varepsilon \mu) w^2 + 2B\varepsilon J w  + 2B \epsilon w v + D\\
			&\leq & (D \lambda - \frac{4}{3}) v^6  + 4v^4+4I v^3 + (D\lambda- 2B\varepsilon \mu) w^2 + 2B\varepsilon J w + ( v^6 + 4 w^2) + B \varepsilon (\mu w^2 + \frac{1}{\mu} v^2) + D \notag
		\end{eqnarray}
		where the last line of \eqref{Vdot2} follows from Cauchy inequality. In addition, the following inequalities 
		\[
		4v^4+4I v^3+ \frac{B \varepsilon}{\mu} v^2 \leq \lambda v^6 + C_\lambda, \quad  2B\varepsilon J w \leq \lambda w^2 + C_\lambda
		\]
		hold for some generic constant $C_\lambda>1$. Since $D$ is independent of $\lambda$, one now chooses 
		\begin{equation}\label{param}
			\lambda := \frac{1}{12D} < 1,\quad B := \frac{6}{\varepsilon \mu},\quad \delta := \frac{1}{4(B+3)} = \frac{\varepsilon \mu}{12(2+\varepsilon \mu)};
		\end{equation}
		then uses \eqref{cauchy} to estimate \eqref{Vdot2} further as 
         \allowdisplaybreaks
		\begin{align}\label{Vdot3}
		&	4 V(z)^3 \langle \nabla V(z), (Id + \psi) f(z+\eta)\rangle
			 \leq  (D \lambda -  \frac{1}{3}) v^6 + (D \lambda + 4- B \varepsilon \mu) w^2 + C_\lambda \notag\\
			\leq & - \frac{1}{4} (\frac{3}{2}v^4-\frac{1}{2}) + (5-B \varepsilon \mu) w^2 + C_\lambda 
			\leq  C_\lambda - \frac{1}{3} v^4 - w^2 \notag\\
			\leq & C_\lambda +\frac{1}{B+3} - \frac{1}{B+3}(1+v^4+Bw^2)
			\leq  C_\lambda - 4 \delta(1+v^4+Bw^2),
		\end{align}
		for a generic constant $C_\lambda$.	Equivalently,
		\[
		\langle \nabla V(z), (Id + \psi) f(z+\eta)\rangle \leq \frac{C_\lambda - 4\delta V(z)^4}{4 V(z)^3} \leq C_\lambda - \delta V(z)
		\]
        which proves \eqref{gradientnegativity}.
	\end{proof}

     \section{Asymptotic dynamics and pullback attractors}    \label{attractorsec}
	Our first main result is the existence and uniqueness theorem, which is formulated as follows.
		
	\begin{theorem}\label{existenceuniqueness}
		Under the assumptions (${\textbf H}_{f}$), (${\textbf H}_{V}$), (${\textbf H}_{g}$) and for any realization $x$ satisfying (${\textbf H}_{X}$), there exists a solution of \eqref{fSDE0} on any interval $[0,T]$. Moreover, the following estimate holds
	\begin{equation}\label{ymucorr}
		V(y_t)   \leq  e^{-\delta t} \Big[V(y_0) - \frac{C_\lambda}{\delta} \Big] + \frac{C_\lambda}{\delta} +H(\ltn \bx \rtn_{\tp,[0,t]}),\quad \forall t \in [0,T]
	\end{equation}	
	where 
	\begin{equation}\label{H}
		H(\xi) = L_V(16C_pC_g)^p \lambda^{1-p}  \xi^p+ 8 L_VC_pC_g \xi.
	\end{equation}		
	\end{theorem}
  \begin{proof}
  	 We are going to prove the theorem under the general condition \eqref{gradientgeneral} in Remark \ref{remLgamma} the estimate 
    \begin{equation}\label{genest}
			V(y_t)\leq \Phi^\gamma (t,V(y_0))+ \max \{1, e^{L_\gamma t}\}H(\ltn \bx \rtn_{\tp,[0,t]})
		\end{equation}
       where we denote by $\Phi^{\gamma}(t,u_0)$ the solution of the scalar ODE $\dot{u} = \gamma(u)$ for one sided Lipschitz function $\gamma$ w.r.t. coefficient $L_\gamma$. Later we simply replace $\gamma (u) := C_\lambda -\delta u$ with $L_\gamma := -\delta <0$ to obtain \eqref{ymucorr}. 

       From the existence and uniqueness theorem for the scalar ODE $\dot{u} = \gamma(u)$, $\Phi^{\gamma}(t,u_0)$ is strictly increasing w.r.t. $u_0$. Moreover, it is easy to check from the one-sided globally Lipschitz continuity of $\gamma$ that 
    \begin{equation}\label{gammaest}
        \Phi^{\gamma}(t,u_0) \leq \Phi^{\gamma}(t,u_0 + \bar{\lambda}) \leq \Phi^{\gamma}(t,u_0) + e^{L_\gamma t} \bar{\lambda},\quad \forall t,\bar{\lambda}\in \R_+,\  \forall u_0 \in \R.
    \end{equation}  
		Given any stopping time interval $[\tau_k,\tau_{k+1}]$, there always exists a unique solution of the equations \eqref{fSDE0} as well as \eqref{ascoODE} on some small interval $[\tau_k,\tau_{\rm local}]$, thus we only need to prove that the solution can be extended into the whole interval $[\tau_k,\tau_{k+1}]$. Considering the transformed system \eqref{ascoODE} on the time interval $[\tau_k,\tau_k \wedge \tau_{\rm local}]$, we deduces from condition \eqref{gradientnegativity} that 
		\begin{eqnarray}\label{Vdot4}
			\dot{V}(z_t)  &=&	\langle \nabla V(z_t), (I + \psi_t) f(z_t+\eta_t)\rangle \notag\\		
			&\leq& 	\sup \limits_{\substack{\psi \in \R^{d\times d},\|\psi\|\leq \lambda \\\eta \in \R^d,\|\eta\| \leq \lambda}} \langle \nabla V(z_t), (I + \psi) f(z_t+\eta)\rangle \notag\\
			&\leq& \gamma (V(z_t)),\quad \forall t \in [\tau_k,\tau_k \wedge \tau_{\rm local}].
		\end{eqnarray}
    Hence by the comparision principle, we obtain 
		\begin{equation}\label{deltalambda}
			V(z_t) \leq \Phi^\gamma (t-\tau_k,V(z_{\tau_k})) = \Phi^\gamma (t-\tau_k,V(y_{\tau_k})),\quad \forall t \in [\tau_k,\tau_k \wedge \tau_{\rm local}].
		\end{equation}
		\eqref{deltalambda} implies that $V(z_t)$ is bounded as long as $t\in [\tau_k,\tau_k \wedge \tau_{\rm local}]$, thereby proving the existence and uniqueness of the solution $z_t$ of equation \eqref{ascoODE} on $[\tau_k,\tau_k \wedge \tau_{\rm local}]$, and so is the solution $y_t$ of \eqref{fSDE0} on $[\tau_k,\tau_k \wedge \tau_{\rm local}]$. In addition, whenever $\tau_{k+1} >  \tau_{\rm local}$ then \eqref{HK} is satisfied and the above arguments can be applied to prove the existence and uniqueness of the solution by concatenation, until the interval $[\tau_k,\tau_{k+1}]$ is fully covered. Moreover, \eqref{deltalambda} and \eqref{lipschitzV} implies that 
		\begin{eqnarray}\label{deltalambda2}
			V(y_{\tau_{k+1}}) = V(z_{\tau_{k+1}} + \eta_{\tau_{k+1}}) \leq V(z_{\tau_{k+1}}) + \|\nabla V\| \|\eta_{\tau_{k+1}}\|
            \leq \Phi^\gamma (\tau_{k+1}-\tau_k,V(y_{\tau_k})) + L_V \lambda. 		
            \end{eqnarray}
		Next for each $t \in [0,T]$, by constructing the sequence of greedy times $\{\tau_k\}=\{\tau_k(\frac{\lambda}{16C_pC_g},\bx,[0,t])\}$ and then applying \eqref{deltalambda2}, \eqref{gammaest}, we can prove by induction that 
        \allowdisplaybreaks
		\begin{eqnarray} \label{Vest2}
			V(y_{\tau_{k+1}}) &\leq& \Phi^\gamma (\tau_{k+1}-\tau_k,V(y_{\tau_k})) + L_V \lambda \notag\\
			&\leq& \Phi^\gamma (\tau_{k+1}-\tau_k,\Phi^\gamma (\tau_k-\tau_{k-1},V(y_{\tau_{k-1}})) + L_V \lambda) + L_V \lambda \notag\\
            &\leq& \Phi^\gamma (\tau_{k+1}-\tau_k,\Phi^\gamma (\tau_k-\tau_{k-1},V(y_{\tau_{k-1}}))) +  e^{L_\gamma (\tau_{k+1}-\tau_k)} L_V \lambda+ L_V \lambda \notag\\
            &\leq& \Phi^\gamma (\tau_{k+1}-\tau_{k-1},V(y_{\tau_{k-1}})) +  2  L_V \lambda \max\{e^{L_\gamma (\tau_{k+1}-\tau_i)}: k-1\leq i \leq k+1\} \notag\\
            &\leq& \ldots \notag\\
            &\leq& \Phi^\gamma (\tau_{k+1}-\tau_0,V(y_{\tau_0})) + (k+1) L_V \lambda \max \{e^{L_\gamma (\tau_{k+1}-\tau_i)}: 0\leq i \leq k+1 \}
		\end{eqnarray}
		for $k = 1,\ldots, N(\frac{\lambda}{16C_pC_g},\bx,[0,t])-1$, where $N(\frac{\lambda}{16C_pC_g},\bx,[0,t]):=\sup \{k \in \N: \tau_k \leq t\}$. It is easy to show (see e.g. \cite{duc21}) the estimates 
	\begin{equation}\label{Nest}
		N(\frac{\lambda}{16C_pC_g},\bx,[0,t]) \leq 1 + \Big(\frac{\lambda}{16C_pC_g}\Big)^{-p} \ltn \bx \rtn^p_{\tp,[0,t]}.
	\end{equation}
    Note also that
    \begin{equation}\label{maxestM}
    \max \{e^{L_\gamma (\tau_k-\tau_i)}: 0\leq i \leq k \leq N(\frac{\lambda}{16C_pC_g},\bx,[0,t]) \} = \max \{1, e^{L_\gamma t}\}.    
    \end{equation}
    As a result, one deduces from \eqref{Vest2}, \eqref{Nest} and \eqref{maxestM} that
		\begin{eqnarray*}
			&&V(y_t)= V(y_{\tau_N}) \\
            &\leq& \Phi^\gamma (t,V(y_0)) + \max \{1, e^{L_\gamma t}\} \Big\{\Big[N(\frac{\lambda}{16C_pC_g},\bx,[0,t])-1\Big]L_V \lambda +  8 C_pC_g L_V\ltn \bx \rtn_{\tp,[\tau_{N-1},t]}\Big\}\\
			&\leq& \Phi^\gamma (t,V(y_0))+ \max \{1, e^{L_\gamma t}\}H(\ltn \bx \rtn_{\tp,[0,t]}).
		\end{eqnarray*}
		This proves \eqref{genest}.
	\end{proof}		
    \begin{remark}\label{remadditive}
    \begin{itemize}
        \item Theorem \ref{existenceuniqueness} and estimate \eqref{ymucorr} still hold under the simplified condition \eqref{addgradientnegativity} for stochastic system \eqref{eq:stochFHN} with additive noise. In fact, we do not need to lift the driving $x$ to a rough path $\bx$, but use it directly to construct the stopping times 
        \begin{equation}\label{greedytimeadditive}
		\tau_0 = 0,\quad \tau_{k+1}:= \inf\Big\{t>\tau_k:  \ltn x \rtn_{\tp, [\tau_k,t]} = \frac{\lambda}{16C_pC_g} \Big\}\wedge T.
	\end{equation}
        Since $\eta$ satisfies \eqref{etaadditive}, it follows that $\|\eta_t\| \leq C_g \ltn x\rtn_{\tp,[\tau_k,\tau_{k+1}]} \leq \lambda$. Then \eqref{Nest} and \eqref{ymucorr} hold for the lift $\bx$ to be replaced by the path $x$. 
        
        \item In addition, using \eqref{Vest2} and \eqref{ymucorr}, one deduces from the local Lipschitz continuity of $f$, the continuity in $(t,\bx,z)$ of $\phi_t(\bx,z), \Big[\frac{\partial \phi }{\partial z}(t,\bx,z)\Big]^{-1}$ (see \cite{duc20}, \cite{duc21}) that both $z_t$ and $y_t$ are continuous solutions w.r.t. $t, \bx$ and initial conditions $z_0, y_0$ respectively. In particular, it is proved in \cite{duckloeden} that there exists a constant $C=C(f,g,\ltn \bx\rtn_{\tp,[0,T]},\|y_0\|)>0$ such that 
		\begin{equation}\label{soludiff}
			\|y_\cdot(\bx,y_0) - y_\cdot(\bx^\pi,y_0)\|_{\infty,[0,T]} \leq C \ltn \bx - \bx^\pi \rtn_{\tp,[0,T]}.
		\end{equation}	
	  \end{itemize}
    \end{remark}
	
	Under the random setting in Appendix \ref{roughpath}, it is proved in \cite{BRSch17}, \cite{duc21} that, on a probability space $(\Omega,\mathcal{F},\mathbb{P})$ equipped with a metric dynamical system $\theta$ (which is assumed to be ergodic in this paper), the rough equation \eqref{eq:stochFHN} and \eqref{fSDE0} generates a random dynamical system $\varphi: \R \times \Omega \times \R^d \to \R^d$, $(t,\omega,y_0)\mapsto \varphi(t,\omega)y_0$ defined by $\varphi(t,\omega)y_0 = y_t(\bx(\omega),y_0)$. 
    Since the solution of \eqref{fSDE0} is continuous w.r.t. $(t,\omega,y_0)$, such $\varphi$ is a continuous mapping w.r.t. $(t,\omega,y_0)$ which satisfies the cocycle property
	\begin{equation}\label{cocycle}
		\varphi(t+s,\omega)y_0 =\varphi(t,\theta_s \omega) \circ \varphi(s,\omega)y_0,\quad \forall t,s \in \R, \omega\in \Omega, y_0 \in \R^d.
	\end{equation} 
	Under assumption (${\textbf H}_{V}$), we aim to use Theorem \ref{existenceuniqueness} to prove that the generated RDS $\varphi$ admits a so-called {\it pullback attractor} $\mathcal{A}$, which is an {\it invariant random set} in the sense $\varphi(t,\omega)\mathcal{A}(\omega) = \mathcal{A}(\theta_t \omega)$ for all $t\in \R,\; \omega\in\Omega$ and attracts any closed tempered random set $\hat{D}$ (i.e. $\limsup \limits_{t \to \infty} \frac{1}{t} \log |D(\theta_{-t}\omega)| =0$ \cite[Section 3.2]{duc21}) in the universe $\mathcal{U}$ in the pullback direction, namely
	\begin{equation}\label{pullback}
		\lim \limits_{t \to \infty} d_H(\varphi(t,\theta_{-t}\omega)
		\hat{D}(\theta_{-t}\omega)| \mathcal{A} (\omega)) = 0,
	\end{equation}
	where $d_H(\cdot|\cdot)$ is the Hausdorff semi-distance. A well-known strategy is to prove the existence of a pullback absorbing set $\cB^*$, which absorbs all closed tempered random sets $\hat{D}$ in the sense that there exists a time $t_0 =t_0(\omega,\hat{D})$ such that
	\begin{equation}\label{absorb}
		\varphi(t,\theta_{-t}\omega)\hat{D}(\theta_{-t}\omega) \subseteq
		\cB^* (\omega), \ \textup{for all}\  t\geq t_0.
	\end{equation}
	The pullback attractor	$\mathcal{A}(\omega)$ can then be proved to exist and can be written as follows
	\begin{equation}\label{at}
		\mathcal{A}(\omega) = \bigcap_{t \geq 0} \overline{\bigcup_{s\geq t} \varphi(s,\theta_{-s}\omega)\cB^*(\theta_{-s}\omega)}. 
	\end{equation}	
We first prove the following estimate.
\begin{proposition}\label{propphiest}
    	Under the assumptions of Theorem \ref{existenceuniqueness}, 
        \begin{equation}\label{V2}
			\alpha(|\varphi(t,\theta_{-t}\omega)D(\theta_{-t}\omega)|) \leq \frac{C_\lambda}{\delta} 
			+ e^{-t\delta } \Big[\beta (|D(\theta_{-t}\omega)|)-\frac{C_\lambda}{\delta}\Big]+\bar{R}_\lambda(\omega),\quad \forall t\in \R_+
		\end{equation}
       for any closed tempered random set $D(\omega)$. 
\end{proposition}
\begin{proof}
    Assume (${\textbf H}^b_{g}$) and (${\textbf H}_{X}$) hold. One rewrites \eqref{ymucorr} as
		\begin{equation*}
			V(\varphi(t,\omega)y_0) - \frac{C_\lambda}{\delta}\leq e^{-\delta t} \Big[V(y_0)-\frac{C_\lambda}{\delta}\Big]+ H(\ltn \bx(\omega) \rtn_{\tp,[0,t]}),
		\end{equation*}
		which follows, by replacing $\omega$ by $\theta_{-t-n}\omega$ for any $n \in \N, t \in [0,1]$ and using \eqref{roughshift} that
		 \allowdisplaybreaks
        \begin{eqnarray*}
			&& V(\varphi(n+t,\theta_{-n-t}\omega)y_0) - \frac{C_\lambda}{\delta} \\
			&=& V(\varphi(1,\theta_{-1}\omega) \varphi(n-1+t,\theta_{-n-t+1} \circ \theta_{-1}\omega)y_0)-\frac{C_\lambda}{\delta}\\
			&\leq& e^{-\delta} [V(\varphi(n-1+t,\theta_{-n-t+1} \circ \theta_{-1}\omega)y_0)-\frac{C_\lambda}{\delta}] + H(\ltn \bx(\theta_{-1}\omega) \rtn_{\tp,[0,1]})\\
			&\leq& e^{-\delta} [V(\varphi(n-1+t,\theta_{-n-t+1} \circ \theta_{-1}\omega)y_0)-\frac{C_\lambda}{\delta}] + H(\ltn \bx(\omega) \rtn_{\tp,[-1,0]}).
		\end{eqnarray*}
		Hence by induction, one can easily prove that for any $n \in \N, t\in [0,1]$ that
		 \allowdisplaybreaks
        \begin{eqnarray*}
			&&   V(\varphi(n+t,\theta_{-n-t}\omega)y_0) - \frac{C_\lambda}{\delta} \\
			&\leq& e^{-n \delta} [V(\varphi(t,\theta_{-t} \circ \theta_{-n}\omega)y_0)-\frac{C_\lambda}{\delta}] + \sum_{k=0}^{n-1} e^{-k\delta}H(\ltn \bx(\theta_{-k}\omega) \rtn_{\tp,[-1,0]})\\
			&\leq& e^{-(n+t) \delta} [V(y_0)-\frac{C_\lambda}{\delta}] + H(\ltn \bx(\theta_{-n}\omega)\rtn_{\tp,[-t,0]})+ \sum_{k=0}^{n-1} e^{-k\delta}H(\ltn \bx(\theta_{-k}\omega) \rtn_{\tp,[-1,0]})\\
			&\leq& e^{-(n+t) \delta} [V(y_0)-\frac{C_\lambda}{\delta}] + \sum_{k=0}^{n} e^{-k\delta}H(\ltn \bx(\theta_{-k}\omega) \rtn_{\tp,[-1,0]}).
		\end{eqnarray*}
		This together with \eqref{Vbound} yield 
		 \allowdisplaybreaks
        \begin{eqnarray*}
        &&\alpha(\|\varphi(n+t,\theta_{-n-t}\omega)\cdot\|_{\infty,D(\theta_{-n-t}\omega)}) - \frac{C_\lambda}{\delta} 
			\leq\|V(\varphi(n+t,\theta_{-n-t}\omega)\cdot)\|_{\infty,D(\theta_{-n-t}\omega)} - \frac{C_\lambda}{\delta} \notag\\
			&\leq& e^{-(n+t)\delta} \Big[\|V(\cdot)\|_{\infty,D(\theta_{-n-t}\omega)}-\frac{C_\lambda}{\delta}\Big]+\sum_{k=0}^{n} e^{-k\delta}H(\ltn \bx(\theta_{-k}\omega) \rtn_{\tp,[-1,0]})\notag\\
			&\leq&e^{-(n+t)\delta } \Big[\beta (|D(\theta_{-n-t}\omega)|)-\frac{C_\lambda}{\delta}\Big]+\bar{R}_\lambda(\omega),
		\end{eqnarray*}
        which proves \eqref{V2}. 
\end{proof}

We need several auxiliary results (the proofs are provided in Appendix).
\begin{lemma}\label{thetaL1}
    Given the measure preserving metric dynamical system $\theta$ on the probability space $(\Omega,\mathcal{F},\mP)$ and any random variable $\xi(\cdot)$ on $\Omega$, it holds that

    i, If $|\xi(\omega)|$ is tempered w.r.t. $\omega$, then $\sup \limits_{t\in \R_+} e^{-\delta t} |\xi(\theta_{-t}\omega)|$ is also tempered  w.r.t. $\omega$ for any $\delta >0$.

    ii, If $|\xi(\cdot)| \in \cL^1$ then $\sup \limits_{t\in \R_+} e^{-\delta t} |\xi(\theta_{-t}\cdot)| \in \cL^1$ for any $\delta >0$. 
    
\end{lemma}

\begin{lemma}\label{Rlem}
    Assign
     \begin{equation}\label{Radius2}
			\bar{R}_\lambda(\omega) :=   \sum_{k=0}^\infty e^{-k\delta} H(\ltn \bx(\theta_{-k}\omega) \rtn_{\tp,[-1,0]}) =   \sum_{k=0}^\infty e^{-k\delta} H(\ltn \bx(\omega) \rtn_{\tp,[-1-k,-k]}),
		\end{equation}
    for $\delta $ in \eqref{gradientnegativity}. Then $\bar{R}_\lambda$ is finite a.s. such that $\bar{R}_\lambda \in \cL^\rho$ for any $\rho \geq 1$, thus $\bar{R}_\lambda$ is also tempered a.s. Moreover, for any $T>0$,
        \begin{equation}\label{Radiusest}
            \max \limits_{t\in [0,T]} \bar{R}_\lambda(\theta_{-t}\omega) \leq \sum_{k=0}^\infty e^{-k\delta} H(\ltn \bx(\omega) \rtn_{\tp,[-T-1-k,-k]})<\infty\quad  \text{a.s.}
        \end{equation}     
\end{lemma}

 \begin{proposition}\label{cBlem}
     The sets
        \begin{eqnarray}
            B_\alpha(\omega)&:=&  B\Big(0,\alpha^{-1} \big(\frac{C_\lambda}{\delta}+ \bar{R}_\lambda(\omega)+\varepsilon\big)\Big); \label{balpha}\\ 
			\bar{\cB}^\varepsilon(\omega) &:=& \overline{\bigcup_{t\geq 0} \varphi(t,\theta_{-t}\omega)B\Big(0,\alpha^{-1} \big(\frac{C_\lambda}{\delta}+ \bar{R}_\lambda(\theta_{-t}\omega)+\varepsilon\big)\Big)},\label{cB2}
		\end{eqnarray}
         for any $\varepsilon > 0$ and any function $\alpha^{-1} \in \mathcal{K}^{\rm tempered}_\infty$ in \eqref{Vbound}, are tempered random set. If in addition $\alpha^{-1},\beta \in \mathcal{K}^{\rm poly}_\infty$, then $|B_\alpha(\cdot)|,|\bar{\cB}^\varepsilon(\cdot)| \in \cL^\rho$ for any $\rho \geq 1$.
 \end{proposition}
 
\begin{proof}
    From Lemma \ref{Rlem}, $\bar{R}_\lambda$ is integrable and tempered a.s. Since $\alpha^{-1}, \bar{R}_\lambda(\theta_{- \cdot}\omega)\in \mathcal{K}^{\rm tempered}_\infty$, it follows that $\alpha^{-1}(\frac{C_\lambda}{\delta} + \bar{R}_\lambda(\theta_{-\cdot}\omega)+\varepsilon)\in \mathcal{K}^{\rm tempered}_\infty$ due to Lemma \ref{Kfunctions}. Hence $B_\alpha(\cdot)$ is a tempered set.

    Next, it follows from \eqref{cB2} and \eqref{V2} and $\alpha^{-1},\beta \circ \alpha^{-1} \in \mathcal{K}^{\rm tempered}_\infty$ that
     \allowdisplaybreaks
    \begin{eqnarray}\label{cBboundexp}
    |\bar{\cB}^\varepsilon(\omega)| &\leq& \sup_{t\in \R_+} |\varphi(t,\theta_{-t}\omega)B_\alpha(\theta_{-t}\omega)|  \notag  \\
    &\leq& \sup_{t\in \R_+} \alpha^{-1} \Big( \frac{C_\lambda}{\delta} 
			+ e^{-t\delta } \Big[\beta (|B_\alpha(\theta_{-t}\omega)|)-\frac{C_\lambda}{\delta}\Big]+\bar{R}_\lambda(\omega) \Big) \notag\\
    &\leq& \sup_{t\in \R_+} \alpha^{-1} \Big( \frac{C_\lambda}{\delta} +\bar{R}_\lambda(\omega)
			+ e^{-t\delta } (\beta\circ \alpha^{-1}) \big(\frac{C_\lambda}{\delta}+\bar{R}_\lambda(\theta_{-t}\omega)+\varepsilon\big) \Big) \notag\\
    &\leq&\alpha^{-1} \Big( \frac{C_\lambda}{\delta} +\bar{R}_\lambda(\omega)
			+  (\beta\circ \alpha^{-1}) \big(\frac{C_\lambda}{\delta}+\sup_{t\in \R_+} e^{-t\delta } \bar{R}_\lambda(\theta_{-t}\omega)+\varepsilon\big) \Big).       
    \end{eqnarray}
    By Lemma \ref{thetaL1}, $\sup\limits_{t\in \R_+} e^{-t\delta } \bar{R}_\lambda(\theta_{-t}\omega)$ is also a tempered random variable. The previous arguments in the first part of the proof for $\alpha^{-1}$ can be applied to prove that $\bar{\cB}^\varepsilon(\omega)$ is a tempered random set. 

     In case $\alpha^{-1}, \beta \in \mathcal{K}^{\rm poly}_\infty$ it follows from \eqref{polyest}, Young's inequality, $\bar{R}_\lambda \in \cL^\rho$ for any $\rho \geq 1$ and Lemma \ref{thetaL1} that 
     \[
     \alpha^{-1}\big(\frac{C_\lambda}{\delta} + \bar{R}_\lambda(\cdot)+\varepsilon\big), \alpha^{-1} \Big( \frac{C_\lambda}{\delta} +\bar{R}_\lambda(\omega)
			+  (\beta\circ \alpha^{-1}) \big(\frac{C_\lambda}{\delta}+\sup_{t\in \R_+} e^{-t\delta } \bar{R}_\lambda(\theta_{-t}\cdot)+\varepsilon\big) \Big) \in \cL^\rho.
    \]
     for any $\rho \geq 1$. This proves $|B_\alpha(\cdot)|,|\bar{\cB}^\varepsilon(\cdot)| \in \cL^\rho$ for any $\rho \geq 1$.\\
\end{proof}
        
	The following result concludes the existence of the global pullback attractor.
	\begin{theorem}\label{attractorgbounded2}
		Under the assumptions of Theorem \ref{existenceuniqueness}, 
        there exists the global pullback attractor $\mathcal{A}$ for the generated random dynamical system $\varphi$ of the stochastic system \eqref{eq:stochFHN} which can be defined by		
        \begin{equation}\label{at4}
			\mathcal{A}(\omega)	= \bigcap_{n=0}^\infty \varphi(n,\theta_{-n}\omega )\bar{\cB}^\varepsilon(\theta_{-n}\omega)= \lim \limits_{n \to \infty} \varphi(n,\theta_{-n}\omega )\bar{\cB}^\varepsilon(\theta_{-n}\omega),	\quad \forall \omega \in \Omega.
		\end{equation}
		where the set $\bar{\cB}^\varepsilon(\omega)$ in \eqref{cB2} is a compact tempered random set which is forward invariant under $\varphi$, i.e. $\varphi(t,\omega)\bar{\cB}^\varepsilon(\omega) = \bar{\cB}^\varepsilon(\theta_t\omega)$ for all $t\in \R_+,\; \omega\in\Omega$. 
        If in addition $\alpha, \beta, \alpha^{-1} \in \mathcal{K}^{\rm poly}_\infty$ then $|\cA(\omega)| \in \cL^\rho$ for any $\rho \geq 1$.
	\end{theorem}
	    \begin{proof}
		It follows from \eqref{V2} that 
        \[
        |\varphi(t,\theta_{-t}\omega)D(\theta_{-t}\omega)| \leq \alpha^{-1}\Big(\frac{C_\lambda}{\delta} 
			+ e^{-t\delta } \Big[\beta (|D(\theta_{-t}\omega)|)-\frac{C_\lambda}{\delta}\Big]+\bar{R}_\lambda(\omega)\Big),\quad \forall t\in \R_+
        \]
        hence any tempered set $D(\cdot)$ is absorbed in the pullback sense into the random ball $B_\alpha(\omega)$ defined in \eqref{balpha} which is compact due to the fact that $\alpha^{-1} \in \mathcal{K}^{\rm tempered}_\infty$. In addition,  it follows from \eqref{V2}, \eqref{Radiusest} and $\alpha^{-1},\beta \in \mathcal{K}^{\rm tempered}_\infty$ that for any interval $[0,T]$,
        \begin{eqnarray*}
            \alpha(|\varphi(t,\theta_{-t}\omega)B_\alpha(\theta_{-t}\omega)|) &\leq&  \frac{C_\lambda}{\delta} 
			+ e^{-t\delta } \Big[\beta (|B_\alpha(\theta_{-t}\omega)|)-\frac{C_\lambda}{\delta}\Big]+\bar{R}_\lambda(\omega)\\
              &\leq& \beta \circ \alpha^{-1} \big(\frac{C_\lambda}{\delta}+\max \limits_{t\in [0,T]}\bar{R}_\lambda(\theta_{-t}\omega) +\varepsilon\big)+\bar{R}_\lambda(\omega)+\frac{C_\lambda}{\delta}.
        \end{eqnarray*}
        This implies
        \begin{equation}\label{BcapT}
       \max \limits_{t\in [0,T]} |\varphi(t,\theta_{-t}\omega)B_\alpha(\theta_{-t}\omega)| \leq \alpha^{-1}\Big(\bar{R}_\lambda(\omega)+ \frac{C_\lambda}{\delta} +(\beta \circ \alpha^{-1}) \big(\frac{C_\lambda}{\delta}+\max \limits_{t\in [0,T]}\bar{R}_\lambda(\theta_{-t}\omega) +\varepsilon\big)\Big)< \infty;
        \end{equation}
        which proves the compactness of $\overline{\bigcup_{t\in [0,T]} \varphi(t,\theta_{-t}\omega)B_\alpha(\theta_{-t}\omega)}$. On the other hand, since $B_\alpha$ is also tempered compact set by Proposition \ref{cBlem}, it is also absorbed in the pullback sense into itself, hence there exists $T(\omega,B_\alpha)>0$ large enough such that 
        \[
        \varphi(t,\theta_{-t}\omega)B_\alpha(\theta_{-t}\omega) \in B_\alpha(\omega),\quad \forall t\geq T(\omega,B_\alpha). 
        \]
        Hence, it follows from \eqref{cB2} that
        \begin{equation}\label{cBupperbound}
       \bar{\cB}^\varepsilon(\omega) \subset \overline{ \bigcup_{t\in [0,T(\omega,B_\alpha)]} \varphi(t,\theta_{-t}\omega)B_\alpha(\theta_{-t}\omega)   }
        \end{equation}
        which proves that $\bar{\cB}^\varepsilon(\omega)$ is a compact set. Its temperedness is followed from Proposition \ref{cBlem}. To prove the forward invariance of $\bar{\cB}^\varepsilon(\omega)$, observe that
         \allowdisplaybreaks
        \begin{eqnarray*}
        \varphi(s,\omega)\bar{\cB}^\varepsilon(\omega) &=& \overline{\bigcup_{t\geq 0} \varphi(s,\omega)\circ \varphi(t,\theta_{-t}\omega)B_\alpha(\theta_{-t}\omega)}\\
            &=&\overline{\bigcup_{t\geq 0} \varphi(s+t,\theta_{-t-s}\circ \theta_s \omega)B_\alpha(\theta_{-t-s}\circ \theta_s\omega)}\\
            &=& \overline{\bigcup_{t\geq s} \varphi(t,\theta_{-t}\circ \theta_s \omega)B_\alpha(\theta_{-t}\circ \theta_s\omega)} \subset \bar{\cB}^\varepsilon(\theta_s\omega),\quad \forall s\in \R_+.
        \end{eqnarray*}      
        As a result, there exists a random pullback attractor $\mathcal{A}$ defined by \eqref{at} w.r.t. $\bar{\cB}^\varepsilon(\omega)$. The forward invariance of $\bar{\cB}^\varepsilon(\omega)$ then proves \eqref{at4}. The conclusion $|\cA(\cdot)| \in \cL^\rho$ is a direct consequence of Proposition \ref{cBlem}.  
	\end{proof}	

    \begin{remark}\label{remradius}
        By choosing $C_g \in (0,\lambda)$, one can write 
        \[
        H(\xi) < C_g \Big[L_V (16C_p\xi)^p + 8 L_VC_p\xi\Big] = C_g H^*(\xi).
        \]
        Hence
        \begin{equation}\label{Rboundexp}
        \bar{R}_{\lambda}(\omega) < C_g R^*(\omega)< \lambda R^*(\omega)\quad \text{where}\quad  R^*(\omega):=\sum_{k=0}^\infty e^{-k\delta} H^*(\ltn \bx(\omega) \rtn_{\tp,[-1-k,-k]}). 
        \end{equation}
        As a result, the right hand side of \eqref{BcapT} is bounded by $\alpha^{-1} \Big(\frac{C_\lambda}{\delta}+1+(\beta \circ \alpha^{-1}) \big(\frac{C_\lambda}{\delta}+1\big) \Big)$, thus the right hand side of \eqref{cBupperbound} is the subset of 
        \[
       \cB_{\alpha,\beta}:= \overline{B\Big(0,\alpha^{-1} \Big(\frac{C_\lambda}{\delta}+1+(\beta \circ \alpha^{-1}) \big(\frac{C_\lambda}{\delta}+1\big) \Big)\Big)}. 
        \]
for $C_g= C_g(T,R^*(\omega))$ small enough. 
        Therefore, $\bar{\cB}^\varepsilon(\omega)$ approaches $\cB_{\alpha,\beta}$ in the sense that
        \begin{equation}\label{limitingset}
     \lim \limits_{C_g \to 0} d_H\Big(\bar{\cB}^\varepsilon (\omega)\Big| \cB_{\alpha,\beta}\Big)=   0 \quad \text{a.s.}
        \end{equation}
    \end{remark}
	
Next, we would like to study the continuity of the attractor w.r.t. the diffusion coefficient $C_g$ as well as the noise approximation. For the first question, denote by $\cA$ and $\cA^0$ are respectively the global attractor of the generated RDS $\varphi$ and the semigroup $\Phi$ of the unperturbed ODE \eqref{eq:abstract}. The following result is a significant improvement of \cite[Theorem 3.4]{duc21} in which the assumptions on global Lipschitz continuity and relative dissipativity of $f$ as well as the uniform convergence to the attractor $\cA^0$ can be omitted, due to \eqref{limitingset}. 
\begin{theorem}\label{attractorgboundedsemicont}
		Under the assumptions of Theorem \ref{existenceuniqueness}, the pullback attractor $\cA$ is upper semi-continuous w.r.t. $C_g$, i.e.
 \begin{equation}\label{attractorsemicont}
     \lim \limits_{C_g \to 0} d_H(\cA(\omega)|\cA^0) = 0\quad \text{a.s.}
 \end{equation}
 If in addition $\alpha,\beta,\alpha^{-1} \in \mathcal{K}^{\rm poly}_\infty$ then 
 \begin{equation}\label{attractorsemicontexpect}
     \lim \limits_{C_g \to 0} \E d_H(\cA(\cdot)|\cA^0) = 0.
 \end{equation}
	\end{theorem}
\begin{proof}
    {\bf Step 1}. First, we would like to prove the result in the simple case that $f$ is globally Lipschitz continuous w.r.t. the Lipschitz constant $C_f$. Following the setting in \cite[Theorem 3.4]{duc21}, we couple each solution $\varphi(t,\omega)y_0$ of \eqref{fSDE0} by a corresponding solution $\Phi(t)y_0$ of the unperturbed equation \eqref{eq:abstract}. Then for any solution $\Phi(t)\mu_0$ starting from a point $\mu_0 \in \cA^0$, it follows from the Lipschitz continuity of $f$ that
    \begin{eqnarray}\label{Phiymu}
\|\Phi(t)y_0-\Phi(t)\mu_0\| &\leq& \|y_0-\mu_0\|e^{C_f t};     \notag   \\
\|\Phi(t)y_0-\Phi(t)\mu_0-\Phi(s)y_0+\Phi(s)\mu_0\| &\leq& \|y_0-\mu_0\|e^{C_f T}(t-s),\quad \forall 0\leq s \leq t \leq T.
    \end{eqnarray}
    By repeating the proof of \cite[Proposition 3.5]{duc21} with some modifications from \eqref{Phiymu}, we obtain the estimate
    \begin{equation}\label{phidiff1}
        \|\varphi(\cdot,\omega)y_0 - \Phi(\cdot)y_0\|_{\infty,[0,T]} \leq C_g \xi(C_f, T,\|f\|_{\infty,\cA^0},\bx(\omega),[0,T])(1+ \|y_0-\mu_0\|^{\frac{2}{p}}),
    \end{equation}
    for all $y_0 \in \R^d,\mu_0\in \cA^0, T \in \R_+,\omega \in \Omega$, where $\xi(C_f, T,\|f\|_{\infty,\cA^0},\bx(\omega),[0,T]) \in \cL^1$ is a generic integrable random variable. Similar to the arguments in the proof of \cite[Theorem 3.4]{duc21}, we apply the Young inequality to the term $C_g \xi(\bx(\omega),[0,T])\|y_0-\mu_0\|^{\frac{2}{p}}$ in the right hand side of \eqref{phidiff1} and deduce that for an arbitrary $\varepsilon \in (0,1)$, there exists a (generic) integrable random varible $\xi(C_f, T,\|f\|_{\infty,\cA^0},\bx(\omega),[0,T],\varepsilon) \in \cL^1$ such that
    \begin{equation}\label{phidiff2}
        \|\varphi(\cdot,\omega)y_0 - \Phi(\cdot)y_0\|_{\infty,[0,T]} \leq \varepsilon \|y_0-\mu_0\| +C_g \xi(C_f, T,\|f\|_{\infty,\cA^0},\bx(\omega),[0,T],\varepsilon),
    \end{equation}
 for all $y_0 \in \R^d,\mu_0\in \cA^0, T \in \R_+,\omega \in \Omega$. It follows from the triangle inequality that
 \[
   \|\varphi(t,\omega)y_0 - z\| \leq \|\Phi(t)y_0-z\| + \varepsilon \|y_0-\mu_0\| +C_g \xi(C_f, T,\|f\|_{\infty,\cA^0},\bx(\omega),[0,T],\varepsilon), 
 \]
for all $t\in [0,T], T\in \R_+, y_0 \in \R^d$ and $z,\mu_0\in \cA^0$. This yields
    \begin{equation}\label{phidiff3}
     d_H(\varphi(T,\omega)y_0|\cA^0) \leq d_H(\Phi(T)y_0|\cA^0)+ \varepsilon d_H(y_0|\cA^0) +C_g \xi(C_f, T,\|f\|_{\infty,\cA^0},\bx(\omega),[0,T],\varepsilon)
    \end{equation}
for all $y_0 \in\R^d, T \in \R_+,\omega \in \Omega$. 

Assign $\cB^1:=\overline{\bigcup_{t\geq 0} \Phi(t)\cB_{\alpha,\beta}}$, then by the same arguments in Theorem \ref{attractorgbounded2}, $\cB^1$ is a compact set and forward invariant under $\Phi$. Because $\cA^0$ is the global attractor of $\Phi$, there also exists for a given $\varepsilon \in (0,1)$ a time $T_{\cB^1,\epsilon}$ such that 
\[
d_H(\Phi(T_{\cB^1,\varepsilon})\cB^1|\cA^0) < \varepsilon. 
\]
We now fix $T:=T_{\cB^1,\varepsilon}$ and $\omega := \theta_{-T_{\cB^1,\varepsilon}}\omega$ and use \eqref{limitingset} to conclude that there exists $C_g  = C_g(\theta_{-T_{\cB^1,\varepsilon}}\omega)\in (0,\lambda)$ small enough such that 
\begin{equation}\label{B1bound}
\cA(\theta_{-T_{\cB^1,\varepsilon}}\omega)\subset \bar{\cB}^\varepsilon(\theta_{-T_{\cB^1,\varepsilon}}\omega) \subset \cB_{\alpha,\beta}\subset \cB^1.   
\end{equation}
Then by applying \eqref{phidiff3} for $\omega := \theta_{-T_{\cB^1,\varepsilon}}\omega, y_0 \in \cA(\theta_{-T_{\cB^1,\varepsilon}}\omega) \subset \cB^1, T:=T_{\cB^1,\varepsilon}$, we obtain from the invariance of $\cA$ and \eqref{B1bound} that
 \allowdisplaybreaks
\begin{equation}\label{estAdiff}
\begin{split}
    d_H(\cA(\omega)|\cA^0) &=d_H(\varphi(T_{\cB^1,\varepsilon},\theta_{-T_{\cB^1,\epsilon}}\omega)\cA(\theta_{-T_{\cB^1,\epsilon}}\omega)|\cA^0) \\
    &= \sup \limits_{y_0 \in \cA(\theta_{-T_{\cB^1,\epsilon}}\omega)} d_H(\varphi(T_{\cB^1,\varepsilon},\omega)y_0|\cA^0) \\
    &\leq \sup \limits_{y_0 \in \cB^1} \Big[d_H(\Phi(T_{\cB^1,\epsilon})y_0|\cA^0)+ \varepsilon d_H(y_0|\cA^0) \Big] \\
    &\qquad+C_g \xi(C_f, T_{\cB^1,\epsilon},\|f\|_{\infty,\cA^0},\bx(\theta_{-T_{\cB^1,\epsilon}}\omega),[0,T_{\cB^1,\epsilon}],\varepsilon) \\
    &\leq d_H(\Phi(T_{\cB^1,\epsilon})\cB^1|\cA^0) + \varepsilon d_H(\cB^1|\cA^0)\\
    &\qquad+C_g \xi(C_f, T_{\cB^1,\epsilon},\|f\|_{\infty,\cA^0},\bx(\omega),[-T_{\cB^1,\epsilon},0],\varepsilon)\\
    &\leq \varepsilon+ \varepsilon d_H(\cB^1|\cA^0) +C_g \xi(C_f, T_{\cB^1,\epsilon},\|f\|_{\infty,\cA^0},\bx(\omega),[-T_{\cB^1,\epsilon},0],\varepsilon)\\
    &\leq 2\varepsilon+ \varepsilon d_H(\cB^1|\cA^0)
\end{split}
\end{equation}
for $C_g  = C_g(\theta_{-T_{\cB^1,\varepsilon}}\omega)\ll 1$ small enough such that $C_g \xi(C_f, T_{\cB^1,\epsilon},\|f\|_{\infty,\cA^0},\bx(\omega),[-T_{\cB^1,\epsilon},0],\varepsilon) < \varepsilon$. As this holds for any given $\varepsilon >0$, \eqref{attractorsemicont} is proved in case $f$ is globally Lipschitz continuous.

{\bf Step 2}. Next, to prove \eqref{attractorsemicont} in case $f$ is only locally Lipschitz continuous, we notice that it suffices to use \eqref{limitingset} to prove \eqref{phidiff1}, which then follows \eqref{phidiff2}, \eqref{phidiff3}, for $T:= T_{\cB^1,\varepsilon}$ and $\omega := \theta_{-T_{\cB^1,\varepsilon}}\omega$ fixed and for $y_0 \in  \bar{\cB}^\varepsilon(\theta_{-T_{\cB^1,\varepsilon}}\omega) \subset \cB^1$, and finally leads to the estimates in \eqref{estAdiff}. Indeed, due to \eqref{V2} and $\bar{\cB}^\varepsilon(\theta_{-T_{\cB^1,\varepsilon}}\omega) \subset \cB^1$,
  \[
       \max \limits_{t\in [0,T_{\cB^1,\varepsilon}]} |\varphi(t,\theta_{-T_{\cB^1,\varepsilon}}\omega)\bar{\cB}^\varepsilon(\theta_{-T_{\cB^1,\varepsilon}}\omega)| \leq \alpha^{-1}\Big(\beta (|\cB^1|)+\bar{R}_\lambda(\omega) \Big) < \infty;
        \]
which yields
\begin{equation}\label{cDbound}
\bigcup_{t \in[0,T_{\cB^1,\varepsilon}] } \varphi(t,\theta_{-T_{\cB^1,\epsilon}}\omega)\bar{\cB}^\varepsilon(\theta_{-T_{\cB^1,\varepsilon}}\omega) \subset \Big\{B\Big(0,\alpha^{-1}\Big(\beta (|\cB^1|)+\bar{R}_\lambda(\omega) \Big)\Big) \bigcup \cB^1\Big\}=:  \cD(\omega).   
\end{equation}
Note that due to the $\Phi$- forward invariance of $\cB^1$, $ \Phi(\cdot)y_0 \in \cB^1$ for any $y_0 \in  \bar{\cB}^\varepsilon(\theta_{-T_{\cB^1,\varepsilon}}\omega) \subset \cB^1$. In that scenerios, the estimates in \eqref{Phiymu} still hold but for $C_f$ is replaced by the local Lipschitz constant $C(f,\cD(\omega))$ and $T$ is replaced by $T_{\cB^1,\epsilon}$. As a result, a similar estimate to \eqref{phidiff1} has the form
\begin{equation}\label{pullbackformest}
\begin{split}
  &\|\varphi(\cdot,\theta_{-T_{\cB^1,\epsilon}}\omega)y_0 - \Phi(\cdot)y_0\|_{\infty,[0,T_{\cB^1,\epsilon}]}\\
  \leq &C_g \xi(C(f,\cD(\omega)), T_{\cB^1,\epsilon},\|f\|_{\infty,\cD(\omega)},\bx(\theta_{-T_{\cB^1,\epsilon}}\omega),[0,T_{\cB^1,\epsilon}])(1+ \|y_0-\mu_0\|^{\frac{2}{p}}), 
\end{split}
 \end{equation}
 which follows similar estimates to \eqref{phidiff2} and \eqref{phidiff3} for such $\xi$ in \eqref{pullbackformest}.
 
 {\bf Step 3}. Finally, in case $\alpha, \beta, \alpha^{-1} \in \mathcal{K}^{\rm poly}_\infty$, it follows from \eqref{cBboundexp} and \eqref{Rboundexp} that for any $C_g \in (0,\lambda)$
 \[
\E |\cA(\cdot)| \leq \E |\bar{\cB}^\varepsilon(\cdot)| \leq \E\alpha^{-1} \Big( \frac{C_\lambda}{\delta} +\lambda R^*(\omega)
			+  (\beta\circ \alpha^{-1}) \big(\frac{C_\lambda}{\delta}+\sup_{t\in \R_+} e^{-t\delta } \lambda R^*(\theta_{-t}\omega)+\varepsilon\big) \Big) < \infty. 
 \]
 The Lebesgue's dominated convergence theorem \cite[Theorem D, pp. 110]{halmos} is then applied to conclude \eqref{attractorsemicontexpect}.
\end{proof}
\begin{remark}
A key point in the proof of Theorem \ref{attractorgboundedsemicont} is the limit \eqref{limitingset} which provides a bound $\cB^1$ for the pullback absorbing set $\bar{\cB}^\varepsilon$ and later derives \eqref{cDbound} by using \eqref{V2}. Since \eqref{limitingset} and \eqref{ymucorr} are derived from Theorem \ref{existenceuniqueness} and Theorem \ref{attractorgbounded2}, they need assumption \eqref{gradientnegativity} for our strong Lyapunov function. There is a chance to derive \eqref{limitingset} from \eqref{phidiff1} and \eqref{phidiff2} using only the classical condition \eqref{classicgradientnegativity}, but for this one needs a global Lipschitz coefficient $C_f$ to obtain \eqref{Phiymu}. 
\end{remark}
\begin{example}\label{FHNstochex}[Fitzhugh Nagumo system revisited] We reconsider the Fitzhugh Nagumo model in Subsection \ref{FHNsec} in the excitable regime where the unique fixed point $\ta^*$ is stable \cite{ducjostdatmarius}. By Theorem \ref{attractorgbounded2} and Theorem \ref{attractorgboundedsemicont}, the stochastic FHN system \eqref{eq:stochFHN} admits a random attractor $\cA(\omega) \in \cL^\rho$ for any $\rho \geq 1$ and $\lim \limits_{C_g \to 0} \cA(\cdot) = \ta^*$ both in the almost sure sense and in the $\cL^\rho$ sense, i.e. $\cA$ is most likely a singleton random pullback attractor. Since $\cA$ is invariant under $\varphi$, $\cA(\theta_t \omega)$ might experience a spike triggered by noise and travel a large excursion into the phase space before returning to the neighborhood of the fixed point $\ta^*$. The frequency of spiking phenomenon decreases as $C_g$ decreases.   
\end{example}

\begin{example}\label{lorenzadditive}[Lorenz system revisited] We reconsider the Lorenz system \eqref{lorenz} under an additive noise and follow Appendix \ref{roughpath} to construct a random dynamical system w.r.t. the additive noise. By Theorem \ref{attractorgbounded2} and Theorem \ref{attractorgboundedsemicont}, the stochastic Lorenz system \eqref{eq:stochFHN} admits a random attractor $\cA(\omega) \in \cL^\rho$ for any $\rho \geq 1$ and $\lim \limits_{C_g \to 0} \cA(\cdot) = \cA^0$ both in the almost sure sense and in the $\cL^\rho$ sense, where $\cA^0$ is the classical Lorenz attractor. 

While it is not easy to find a strong Lyapunov function for Lorenz system with general multiplicative noise, note that the existence of the random Lorenz attractor is proved in \cite{schmalfuss97} for the multiplicative Stratonovich noise of the form $\chi y \circ dB$ for a scalar standard Brownian motion and a constant $\chi$. Here in our setting, we can easily extend the results to the form $(\chi y +\bar{g})dx$ for a scalar realization $x$ of a scalar stochastic process $X$ satisfying (${\textbf H}_X$), and define $C_g := |\chi| \vee |\bar{g}|$. Since we can solve the pure rough differential equation $d \phi = (\chi \phi + \bar{g})dx$ for an explicit solution, we obtain the explicit Doss-Sussmann transformation as follows 
\[
\xi_t = \phi(t,\tau_k,\bx,\bar{\xi}_t)= e^{\chi x_{\tau_k,t}} \bar{\xi}_t + \bar{g} \int_{\tau_k}^t e^{\chi x_{s,t}}dx_s = \frac{1}{1+\psi_t}\bar{\xi}_t + \eta_t,\quad \forall t\in [\tau_k,\tau_{k+1}]. 
\]
The transformed ODE has the form $\dot{\bar{\xi}}_t = (1+\psi_t)f(\frac{1}{1+\psi_t}\bar{\xi}_t + \eta_t)$. With the same Lyapunov function $V(\xi)$ as in \eqref{Lyalorenz}, a direct computation shows that, for all $|\psi|,\|\eta\| \leq \lambda$,
\begin{eqnarray}\label{Vproplorenz}
  &&V(\frac{1}{1+\psi}\bar{\xi} + \eta) \leq V(\frac{1}{1+\psi}\bar{\xi}) + L_V \|\eta\| \leq (1+\lambda) V(\bar{\xi}) + L_V \lambda; \notag\\
&&\|f(\xi + \eta) - f(\xi)\|\leq \lambda \max \{\sigma, \rho,\beta,1\} (1+\|\xi\|);\notag\\
&&\langle \nabla V(\bar{\xi}),(1+\psi)f(\frac{1}{1+\psi}\bar{\xi} + \eta) \rangle \leq C_\lambda - \delta V(\bar{\xi});
\end{eqnarray}
where the last inequality holds for certain constants $C_\lambda, \delta >0$. As a result, the estimate \eqref{ymucorr} is still correct upto a term $(1+\lambda)^{N(\frac{\lambda}{16C_pC_g},\bx,[0,t])} \leq \exp \{\lambda N(\frac{\lambda}{16C_pC_g},\bx,[0,t])\}$ in the right hand side. As a result, the conclusions in Theorem \ref{attractorgbounded2} and Theorem \ref{attractorgboundedsemicont} still hold with some minor modifications in the proof, where one can choose $\lambda:= C_g$ for $C_g$ sufficiently small.
\end{example}

Regarding to noise approximation, we would like to consider for any rough path $\omega \in \Omega$ the dyadic approximation $\{\pi^{(n)}(\omega)\}_{n \in \N} $ with 
	\[
	x(\pi^{(n)}(\omega))_{t} = x(\omega)_{\frac{k}{2^n}} + (2^n t-k) x(\omega)_{\frac{k}{2^n},\frac{k+1}{2^n}},\quad \forall t \in \Big[\frac{k}{2^n},\frac{k+1}{2^n}\Big], k \in \Z
	\] 
	and its natural lift $\mathcal{S}(x(\pi^{(n)}(\omega)))$ where
	\[
	\mathcal{S}(x)_{s,t} = (1,x_{s,t},\int_s^t x_{s,r}dx_r),\quad \forall s\leq t, s,t \in \R.
	\]
	It follows from \cite[Theorem 5.20]{friz} that 
	\begin{equation}\label{uniformbounded}
		\begin{split}
			\sup_{n \in \N} \ltn x(\pi^{(n)}(\omega)) \rtn_{\tp,[a,b]} & \leq 3^{1-\frac{1}{p}} \ltn x(\omega) \rtn_{\tp,[a,b]},\\ 	
			\sup_{n \in \N} \ltn x(\pi^{(n)}(\omega)) \rtn_{\alpha,[a,b]} & \leq 3^{1-\alpha} \ltn x(\omega) \rtn_{\alpha,[a,b]},\quad \forall a\leq b, a,b \in \N. 
		\end{split}
	\end{equation}
	By the Borel-Cantelli lemma and \cite[Chapter 12]{friz}, it follows that for almost sure all $\omega \in \Omega$
	\[
	\lim \limits_{n \to \infty} d_{\tp,[a,b]} (\pi^{(n)}(\omega) ,\omega) = 	\lim \limits_{n \to \infty} d_{\alpha,[a,b]} (\pi^{(n)}(\omega) ,\omega)= 0,\quad \forall a\leq b, a, b \in \N.
	\]
	The definition of $\pi^{(n)}(\omega)$ yields
	\begin{equation}\label{dyadicshift}
		\Big(\theta_{-i} \pi^{(n)}(\omega)\Big)_{\frac{k}{2^n}} = \Big(\pi^{(n)}(\theta_{-i}\omega)\Big)_{\frac{k}{2^n}} =\Big(\theta_{-i} \omega\Big)_{\frac{k}{2^n}},\quad \forall i, n \in \N, k \in \Z.
	\end{equation}
	As proved in Theorem \ref{attractorgbounded2}, there exist random sets $\bar{\cB}(\omega),\mathcal{A}(\omega)$ and $\bar{\cB}(\bx^{(n)}(\omega)), \mathcal{A}(\bx^{(n)}(\omega))$ for $\omega \in \Omega$. We conclude the upper semi-continuity of the attractor with respect to the dyadic approximation as follows.
	\begin{theorem}\label{dyadicconvergence}
		Under the assumptions of Theorem \ref{existenceuniqueness}, the following limit holds
		\begin{equation}\label{Alim}
			\lim \limits_{n \to \infty} d_H(\mathcal{A}(\pi^{(n)}(\omega))|\mathcal{A}(\omega)) = 0\quad a.s.
		\end{equation}
         If in addition $\alpha,\beta,\alpha^{-1} \in \mathcal{K}^{\rm poly}_\infty$ then 
 \begin{equation}\label{Alimexpect}
     \lim \limits_{C_g \to 0} \E d_H(\mathcal{A}(\pi^{(n)}(\cdot))|\mathcal{A}(\cdot)) = 0.
 \end{equation}
	\end{theorem}
	
    \begin{proof}
		First it follows from \eqref{Radius2} and \eqref{uniformbounded} that 
		\begin{equation}\label{Apiest}
		\sup_{n\in \N}\bar{R}_\lambda(\pi^{(n)}(\omega)) \leq \Gamma_p \bar{R}_\lambda(\omega), \quad \forall \omega\in \Omega
		\end{equation}
		for some $\Gamma_p >1$, so that 
		\begin{equation}\label{unibound}
			\bar{\cB}(\pi^{(n)}(\omega)) = V^{-1}(\frac{C_\lambda}{\delta} + \bar{R}_\lambda(\pi^{(n)}(\omega)))  \subset V^{-1}(\frac{C_\lambda}{\delta} + \Gamma_p \bar{R}_\lambda(\omega)) = \cB^*(\omega),\quad \forall n \in \N.
		\end{equation}
		On the other hand, for any $y_0 \in \cB^*(\omega)$
		\[
		V(\varphi(1,\omega)y_0) - \frac{C_\lambda}{\delta} \leq e^{-\delta} \Gamma_p \bar{R}_\lambda(\omega) + H(\ltn \bx(\theta_1\omega)\rtn_{\tp,[-1,0]}) \leq \Gamma_p \bar{R}_\lambda (\theta_1 \omega),
		\]
		which implies that $\cB^*$ is forward invariant, i.e. $\varphi(1,\omega)\cB^*(\omega) \subset \cB^*(\theta_1 \omega)$. Similar to Theorem \ref{attractorgbounded2}, one can use \eqref{dyadicshift} to prove that for all $n \in \N, \omega \in \Omega$
		\begin{equation}\label{at5}
			\begin{split}
				\mathcal{A}(\pi^{(n)}(\omega))	&=  \lim \limits_{i \to \infty} \varphi(i,\pi^{(n)}(\theta_{-i}\omega) )\bar{\cB}(\pi^{(n)}(\theta_{-i}\omega)) =\bigcap_{i=1}^\infty \varphi(i,\pi^{(n)}(\theta_{-i}\omega) )\bar{\cB}(\pi^{(n)}(\theta_{-i}\omega)),\\	
				\mathcal{A}(\omega)	&=  \lim \limits_{i \to \infty} \varphi(i,\theta_{-i}\omega )\bar{\cB}^*(\theta_{-i}\omega) =\bigcap_{i=1}^\infty \varphi(i,\theta_{-i}\omega )\bar{\cB}^*(\theta_{-i}\omega).		
			\end{split}	
		\end{equation}
		As a result, for any $\delta >0$ there exists $T_\delta(\omega) \in \N$ such that 
		\begin{equation}\label{triangle1}
			d(\varphi(i,\theta_{-i}\omega)\bar{\cB}^*(\theta_{-i}\omega)\ | \mathcal{A}(\omega))  < \delta,\quad \forall i \geq T_\delta(\omega). 
		\end{equation}
		In addition, it follows from \eqref{unibound} that
		\begin{equation}\label{triangle2}
			\mathcal{A}(\pi^{(n)}(\omega))	\subset \bigcap_{i=1}^\infty \varphi(i,\pi^{(n)}(\theta_{-i}\omega) )\bar{\cB}^*(\theta_{-i}\omega),\quad \forall n\in \N.
		\end{equation}
		Now for any fixed $i \geq T_\delta(\omega)$, it follows from \eqref{soludiff} that 
		\begin{eqnarray}\label{triangle3}
			&& d(\varphi(i,\pi^{(n)}(\theta_{-i}\omega) )\bar{\cB}^*(\theta_{-i}\omega)\ |  \varphi(i,\theta_{-i}\omega)\bar{\cB}^*(\theta_{-i}\omega)) \notag\\
			&\leq& C(f,g,\ltn \bx(\theta_{-i}\omega) \rtn_{\tp,[0,i]},|\bar{\cB}^*(\theta_{-i}\omega)|) d_{\tp,[-i,0]} (\pi^{(n)}(\omega) ,\omega) < \delta,
		\end{eqnarray}
		for all $n \geq N(\delta,i,\omega)$, where $N(\delta,i,\omega)$ is large enough and $C(f,g,\ltn \bx(\theta_{-i}\omega) \rtn_{\tp,[0,i]},|\bar{\cB}^*(\theta_{-i}\omega)|)$ is a certain constant from \eqref{soludiff}. The estimates \eqref{triangle1}, \eqref{triangle2}, \eqref{triangle3} and the triangle inequality then prove $d_H(\mathcal{A}(\bx^{(n)}(\omega))\ | \mathcal{A}(\omega)) < 2 \delta$ for all $n \geq N(\delta,i,\omega)$, which proves \eqref{Alim}. 
        
        Once again, the conclusion \eqref{Alimexpect} is directly followed from \eqref{Apiest}, \eqref{Rboundexp} and the Lebesgue's dominated convergence theorem. 
    \end{proof}

   \subsection{Numerical attractors}
   In this section, we study the dynamics of the discrete system 
		\begin{equation}\label{REuler}
			\begin{split}
				y^\Delta_0 &\in \R^d,\\
				y^\Delta_{t_{k+1}} &= y^\Delta_{t_k} + f(y^\Delta_{t_k}) \Delta + g(y^\Delta_{t_k})x_{t_k, t_{k+1}} + Dg(y^\Delta_{t_k})g(y^\Delta_{t_k})\X_{t_k,t_{k+1}},\quad k \in \N.
			\end{split}
        \end{equation}
under the regular grid $\Pi = \{t_k:=k\Delta\}_{k \in \N}$, $0<\Delta\leq 1$. The global dynamics of the discrete system \eqref{REuler} has been studied recently in \cite{duckloeden}, \cite{congduchong23} for dissipative drift $f$ that is globally Lipschitz continuous, which show that the existing random pullback attractor of the discrete system \eqref{REuler} converges to the random attractor of the continuous system \eqref{fSDE0} as the step size $\Delta$ tends to zero. A difficulty in dealing with the discrete system is that we can not apply the Doss-Sussmann technique, simply because it is difficult to control the solution growth in a smooth way for the discrete time set. Fortunately, we still can apply the coupling technique in the proof of Theorem \ref{attractorgboundedsemicont} for the discrete system \eqref{REuler} and use only classical Lyapunov functions to conclude on the numerical attractors. This can be stated in the following theorem.
\begin{theorem}\label{attractordisc1}
    Under the assumptions (${\textbf H}_{V}$), (${\textbf H}_{g}$), assume $f$ is globally Lipschitz continuous with Lipschitz constant $C_f$ and there exists a classical Lyapunov function $V \in \cC^2$ satisfying \eqref{classicgradientnegativity} and further 
    \begin{equation}\label{Vdiscrete}
    \|D^2 V\|_{\infty,\R^d} <\infty; \quad  \alpha \|z\|^2 - C\leq V(z) \leq \beta (\|z\|^2+1), \quad \forall z \in \R^d
    \end{equation}
    for certain constants $\alpha,\beta >0, C\in \R$. Then there exists for sufficiently small step size $\Delta$ a random pullback attractor $\cA^\Delta$ for the discrete RDS $\varphi^\Delta$ generated by the discrete system \eqref{REuler}. For $C_g$ small enough, $\cA^\Delta$ approximates the continuous attractor $\cA$ of the continuous RDS $\varphi$, i.e.
     \begin{equation}\label{attractorsemicontdiscexpect}
     \lim \limits_{\Delta \to 0} d_H(\cA^\Delta(\omega)|\cA) = 0\quad \text{a.s. and }\quad 
     \lim \limits_{\Delta \to 0} \E d_H(\cA^\Delta(\cdot)|\cA) = 0. 
 \end{equation}
 Moreover, the discrete pullback attractor $\cA^\Delta$ is upper semi-continuous w.r.t. $C_g$, i.e.
 \begin{equation}\label{attractorsemicontdiscCgexpect}
 \lim \limits_{C_g \to 0} d_H(\cA^\Delta(\omega)|\cA^{\Delta,0}) = 0\quad \text{a.s. and} 
   \lim \limits_{C_g \to 0} \E d_H(\cA^\Delta(\cdot)|\cA^{\Delta,0}) = 0.
 \end{equation}
\end{theorem}
\begin{proof}
   Condition \eqref{Vdiscrete} implies that $V$ satisfies \eqref{Vbound} for bound functions in $\mathcal{K}_\infty^{\rm poly}$.   
   The proof follows the arguments in \cite[Theorem 5.3]{congduchong23} and \cite[Proposition 5.7]{duchongcong24}. Denote by $\Phi^\Delta$ the discrete semigroup generated by the unperturbed discrete system  
    \begin{equation}\label{unperturbdisc}
    \mu^\Delta_{t_{k+1}} = \mu^\Delta_{t_k} + f(\mu^\Delta_{t_k}) \Delta,\quad t_k \in \Pi.     
    \end{equation}
    Then it follows from \eqref{classicgradientnegativity}, \eqref{Vbounds} and the Taylor expansion that 
     \allowdisplaybreaks
    \begin{eqnarray}\label{Vdeterministic}
        V(\mu^\Delta_{k+1}) &\leq& V(\mu^\Delta_k) + \langle \nabla V(\mu^\Delta_k),f(\mu^\Delta_k)\rangle \Delta + \frac{1}{2}\|D^2 V\|_{\infty,\R^d} \|f(\mu^\Delta_k)\|^2 \Delta^2 \notag\\
        &\leq& V(\mu^\Delta_k) (1-d_2\Delta) + d_1\Delta +\|D^2 V\|_{\infty,\R^d} (C_f^2\|\mu^\Delta_k\|^2+ \|f(0)\|^2) \Delta^2 \notag\\
        &\leq& V(\mu^\Delta_k) \Big(1-d_2\Delta + \|D^2 V\|_{\infty,\R^d} \frac{C_f^2}{\alpha} \Delta^2\Big) + d_1\Delta + \|D^2 V\|_{\infty,\R^d}\|f(0)\|^2 \Delta^2 \notag\\
        &\leq& V(\mu^\Delta_k) (1-\frac{d_2}{2}\Delta) + C(f,V) \Delta
    \end{eqnarray}
    for a constant $C(f,V)$ and a sufficiently small step size $\Delta$ satisfying
    \[
   C(f,V):= d_1 + \|D^2 V\|_{\infty,\R^d}\|f(0)\|^2 \Delta;\quad  0<\Delta \leq \Delta_0:=\frac{\alpha d_2}{2C_f^2 \|D^2 V\|_{\infty,\R^d}}.
    \]
As a result, it follows from \eqref{Vdeterministic} and \eqref{Vdiscrete} that
\begin{eqnarray}\label{discphiA0}
\|\Phi^\Delta(t_k)\mu^\Delta_0\| =\|\mu^\Delta_k\| &\leq& \Big \{\frac{1}{\alpha} V(\mu_k^\Delta)+\frac{1}{\alpha}C \Big\}^{\frac{1}{2}}
 \notag\\
&\leq&\Big \{ \Big(1-\frac{d_2}{2}\Delta\Big)^k \frac{\beta}{\alpha} \|\mu^\Delta_0\|^2 + \frac{1}{\alpha}\Big(\beta+\frac{2}{d_2} C(f,V)+C\Big)\Big\}^{\frac{1}{2}} \notag\\
&\leq&e^{-\frac{d_2}{2}t_k} \sqrt{\frac{\beta}{\alpha}} \|\mu^\Delta_0\| + \underbrace{\sqrt{\frac{1}{\alpha}\Big(\beta+\frac{2}{d_2} C(f,V)+C\Big)}}_{=: \bar{C}},\quad \forall k \in \N
\end{eqnarray}
thus there exists a discrete attractor $\cA^{\Delta,0}$ of the unperturbed discrete system \eqref{unperturbdisc}. Since $f$ is globally Lipschitz continuous, we obtain similar estimates to \eqref{Phiymu} that for any $T \in \Pi$
\begin{equation}\label{Phideltaest}
\begin{split}
 \ltn \Phi^\Delta(\cdot)\mu_0^\Delta\rtn_{1{\rm- var},\Pi[0,T]} &\leq T \|f(\mu^\Delta)\|_{\infty,\Pi[0,T]} \leq T\Big[\|f(0)\|+C_f\bar{C}  (1+ \|\mu_0^\Delta\|)\Big];\\
 \| \Phi^\Delta(\cdot)y_0^\Delta-\Phi^\Delta(\cdot)\mu_0^\Delta\|_{\infty,\Pi[0,T]} &\leq (1+C_f \Delta)^{\frac{T}{\Delta}} \|y_0^\Delta-\mu_0^\Delta\| \leq e^{C_fT}\|y_0^\Delta-\mu_0^\Delta\|;\\
 \ltn \Phi^\Delta(\cdot)y_0^\Delta-\Phi^\Delta(\cdot)\mu_0^\Delta\rtn_{1-{\rm var},\Pi[s,t]}
 &\leq (t-s)e^{C_fT}\|y_0^\Delta-\mu_0^\Delta\|,\quad \forall 0 \leq s \leq t \leq T.
 \end{split}    
\end{equation}
Similar to \eqref{phidiff1}, \eqref{phidiff2} in the continuous case, we show in the Appendix that there exists a (generic) integrable random variable $ \xi(C_f, T,\|f\|_{\infty,\cA^{\Delta,0}},\bx(\omega),[0,T])$ such that for any arbitrary $\varepsilon < 1$ fixed and any $\mu_0^\Delta \in \cA^{\Delta,0}$
\begin{eqnarray}\label{hestdiscrete}
   && \|\varphi^\Delta (\cdot,\omega)y_0^\Delta - \Phi^\Delta(\cdot)y_0^\Delta\|_{\infty,\Pi[0,T]} \notag\\
   &\leq& C_g \xi(C_f, T,\|f\|_{\infty,\cA^{\Delta,0}},\bx(\omega),[0,T])(1+ \|y^\Delta_0-\mu^\Delta_0\|^{\frac{2}{p}}) \notag\\
    &\leq& \varepsilon \|y_0^\Delta-\mu^\Delta_0\| +C_g\xi(C_f, T,\|f\|_{\infty,\cA^{\Delta,0}},\bx(\omega),[0,T]),\quad \forall y_0^\Delta \in \R^d, \forall T \in \R_+,\forall \omega \in \Omega. 
\end{eqnarray}
Therefore, fix a sufficiently large $T\in \Pi$ such that there exists a constant $\delta >0$ with 
\[
e^{-\delta}=e^{-\frac{d_2}{2}\lfloor T \rfloor}\sqrt{\frac{\beta}{\alpha}} +\varepsilon <1, 
\]
then it follows from \eqref{discphiA0} that the discrete attractor $\cA^{\Delta,0}$ is a subset of an absorbing set, i.e. 
\[
\cA^{\Delta,0} \subset \cB^0 = B\Big(0,1+\frac{\bar{C}}{1-e^{-\delta}}\Big) := \Big\{\mu_0 \in \R^d: \|\mu_0\|\leq 1+\frac{\bar{C}}{1-e^{-\delta}}\Big\},\quad \forall \Delta \leq \Delta_0.
\] 
This, together with \eqref{hestdiscrete}, yields
\begin{eqnarray}\label{phidelta1}
    \|\varphi^\Delta (T,\omega)y_0^\Delta\| 
    &\leq& \|\Phi^\Delta(T)y_0^\Delta\| + \varepsilon \|y_0^\Delta\| +C_g\xi(C_f, T,\|f\|_{\infty,\cA^{\Delta,0}},\bx(\omega),[0,T])\notag\\
    &\leq&\Big(e^{-\frac{d_2}{2}T}\sqrt{\frac{\beta}{\alpha}} +\varepsilon\Big) \|y_0^\Delta\| + \bar{C}+C_g\xi(C_f, \lceil T \rceil,\|f\|_{\infty,\cB^0},\bx(\omega),[0,\lceil T\rceil])\notag\\
    &\leq& e^{-\delta} \|y_0^\Delta\| +\bar{C}+ C_g\xi(C_f, \lceil T\rceil,\|f\|_{\infty,\cB^0},\bx(\omega),[0,\lceil T\rceil]).
\end{eqnarray}
Observe that for $\Delta \leq \Delta_0$ small enough, $\delta, \cB^0, \lceil T \rceil$ are independent of $\Delta$.
Estimate \eqref{phidelta1} then enables us to follow the arguments in \cite[Section 5]{congduchong23} (and similar to Theorem \ref{attractorgbounded2}) to prove that the discrete RDS $\varphi^\Delta$ admits a discrete random pullback attractor $\cA^\Delta$, which is uniformly bounded in $\Delta$ for sufficiently small $C_g$. In particular, there exists a discrete pullback absorbing set $\bar{\cB}^{\varepsilon,\Delta} \supset \cA^\Delta$ which is forward invariant and uniformly bounded by $\cB^0$, i.e. the similar limit to \eqref{limitingset} holds
\begin{equation}\label{uniformboundeddiscrete}
 \lim \limits_{C_g \to 0} d_H\Big(\bar{\cB}^{\varepsilon,\Delta} (\omega)\Big| \cB^0\Big)=   0 \quad \text{a.s.}
\end{equation}
The almost sure convergence in \eqref{attractorsemicontdiscexpect} is thus a direct consequence of \cite[Theorem 5.3]{congduchong23}. The expectation convergence in \eqref{attractorsemicontdiscexpect} comes from Lebesgue's dominated convergence theorem.

Next, estimate \eqref{hestdiscrete} yields a similar estimate to \eqref{phidiff3} that
     \begin{eqnarray*}
     d_H(\varphi^\Delta(T,\omega)y^\Delta_0|\cA^{\Delta,0}) &\leq& d_H(\Phi^\Delta(T)y^\Delta_0|\cA^{\Delta,0})+ \varepsilon d_H(y_0|\cA^{\Delta,0})\\
     &&+C_g \xi(C_f, \lceil T\rceil,\|f\|_{\infty,\cB^0},\bx(\omega),[0,\lceil T\rceil],\varepsilon)
    \end{eqnarray*}
for all $y^\Delta_0 \in\R^d, T \in \R_+,\omega \in \Omega$. From this step, the proof of the convergence in \eqref{attractorsemicontdiscCgexpect} follows the arguments of Theorem \ref{attractorgboundedsemicont} line by line, taking into account \eqref{uniformboundeddiscrete}.
\end{proof}

\subsection{Local stability}\label{localstabsec}
In this subsection, we assume that $f$ is locally Lipschitz continuous on a compact domain $\cD$ with constant $C_f(\cD)$ and a local Lyapunov function $V$ satisfies \eqref{classicgradientnegativity} on $\cD$ such that $V^{-1}(\frac{d_1}{d_2}) \subset \cD$.
In the local setting, any solution starting from a point $y_0$ near the boundary of $\cD$ might leave $\cD$ at a certain time on a fixed time interval $[0, T ]$, as it is a random variable. 

        As a special case, we are able to investigate the local stability of system \eqref{fSDE0} in a compact vicinity $\mathcal{D}$ of the trivial solution, where we assume for simplicity that $f(0) =0$, while $g\in C^3_b$ with $g(0)=0$. In this setting, we can apply the arguments in \cite{duchongcong24} to prove that the trivial solution of \eqref{fSDE0} are still locally exponentially stable in a random sub-domain of $\mathcal{D}$. 
		Our target is to find a Lyapunov function $V$ of the unperturbed autonomous system such that there exists constants $\alpha, \beta >0$ such that \eqref{Vbounds} is satisfied with $C:=V(0)$ and \eqref{classicgradientnegativity} has the form
		\begin{eqnarray}
			\langle \nabla V(z), f(z)\rangle &\leq& - d_2 [V(z)-V(0)],\quad \forall z \in \cD. \label{gradVmodified}
		\end{eqnarray}
       Motivated by \cite{duchongcong24}, we conclude the following result on local stability of the trivial solution.
    \begin{theorem}\label{localstabLya}
        If there exists a Lyapunov function $V$ satisfying \eqref{Vbounds}, \eqref{lipschitzV} and \eqref{gradVmodified} then there exists a $\lambda_0>0$ small enough such that for any $C_g < \lambda_0$, the trivial solution of system \eqref{fSDE0} is locally exponentially stable on $\mathcal{D}$.
    \end{theorem}
  
    \begin{proof}
   Since $g(0)=0$, it follows from \cite[Proposition 2.1]{duchongcong24} that the first estimate for the solution $\phi$ of the pure rough equation \eqref{pure} is modified to
		\begin{equation}\label{solest1} 
			\ltn \phi \rtn_{\tp,[a,b]} \leq 8 C_p C_g \ltn \bx \rtn_{\tp,[a,b]}\min \left\{\|\phi_a\|, 1\right\}.
		\end{equation}		
		Given the Doss-Sussmann transformation $y_t = \phi_t(\bx,z_t)$, the difference $\eta_t := y_t - z_t$ on each time interval $16 C_p C_g \ltn \bx \rtn_{\tp,[0,\tau]} \leq \lambda <1$ then satisfies $\|\eta_t\| \leq \lambda \|z_t\|$ for all $t \in [0,\tau]$.    
        Also, since $f$ is locally Lipschitz continuous on $\mathcal{D}$ with $f(0) =0$, it yields
        \[
        \|f(z+\eta)\| \leq C_f(\cD) (\|z\|+ \|\eta\|) \quad \text{and}\quad \|f(z+\eta)-f(z)\| \leq C_f(\cD) \|\eta\|.
         \]
        In this scenarios, the estimate \eqref{Lipschitzfestimate} is re-written as follows
         \allowdisplaybreaks
              \begin{eqnarray*}
			&&\langle \nabla V(z), (I + \psi)f(z+\eta) \rangle \notag\\
            &=& \langle \nabla V(z), f(z) \rangle + \langle \nabla V(z), \psi f(z+\eta) \rangle + \langle \nabla V(z), f(z+\eta) - f(z) \rangle \notag\\
			&\leq& - d_2 [V(z)-V(0)] + \|\nabla V(z)\| \Big(\|\psi\| \|f(z+\eta)\| + \|f(z+\eta)-f(z)\|\Big)  \notag\\
			&\leq& - d_2 [V(z)-V(0)] + L_V \Big[\|\psi\|C_f(\cD) (\|z\|+\|\eta\|) + C_f(\cD) \|\eta\| \Big] \notag\\
			&\leq& - d_2 [V(z)-V(0)] + L_VC_f(\cD)\lambda (2+\lambda) \|z_t\|\notag\\
			&\leq& \Big[\frac{1}{\alpha}C_f(\cD)L_V(2+\lambda)\lambda - d_2\Big] [V(z)-V(0)] \\
            &\leq& -\delta [V(z)-V(0)] 
		\end{eqnarray*}
        where we choose $\lambda < (\frac{\alpha d_2}{3C_f(\cD)L_V} \wedge 1)$ such that $\delta:= d_2 - \frac{1}{\alpha}C_f(\cD)L_V(2+\lambda)\lambda >0$. Hence \eqref{deltalambda2} has the form
        \begin{eqnarray*}
          V(y_\tau)-V(0) &\leq& V(z_\tau)-V(0) + L_V \|\eta_\tau\| \\
          &\leq& V(z_\tau)-V(0) + L_V \lambda \|z_\tau\| \leq (1+\frac{L_V\lambda}{\alpha})[V(z_\tau)-V(0)] \\
          &\leq& (1+\frac{L_V\lambda}{\alpha}) e^{-\delta \tau} [V(y_0)-V(0)] \leq \exp \{-\delta \tau + \frac{L_V\lambda}{\alpha}\}[V(y_0)-V(0)] .            \end{eqnarray*}
        Hence similar to the proof of Theorem \ref{existenceuniqueness}, we can prove by induction that
        \begin{equation}
            [V(y_1)-V(0)] \leq \exp \Big\{-\delta +  \frac{L_V\lambda}{\alpha} N(\frac{\lambda}{16C_pC_g},\bx,[0,1]) \Big\} [V(y_0)-V(0)].
        \end{equation}
        Next, we follow the arguments in \cite[Theorem 4.5, Theorem 4.10]{duchongcong24} by choosing $0<C_g<\lambda <\lambda_0< (\frac{d_2}{3C_f(\cD)L_V} \wedge 1)$ for sufficiently small $\lambda_0$ that
        \begin{equation}
            \mu = d_2 - C_f(\cD)L_V(2+\lambda_0)\lambda_0 - \frac{L_V\lambda_0}{\alpha} \E N\Big(\frac{1}{16C_p},\bx,[0,1]\Big) >0.
        \end{equation}
        With such choice of $C_g$ and $\lambda$, there exists for any $\epsilon >0$ random radii $r(\omega), r^*(\omega) < \epsilon$ such that for any $\|y_0\|\leq r(\omega) <\epsilon$ in $\mathcal{D}$ then $\|y_t\|<\epsilon$ for all $t\geq 0$ and further
        \[
        \|y_t\| \leq r^*(\omega) e^{-\mu t},\quad \forall t\geq 0.
        \]
        This proves the locally exponential stability of the trivial solution.\\
        \end{proof}

	\begin{example}
	    Continuing Example \ref{penex}, consider new Lyapunov function 
        \[
        V(z) := \Big[\frac{1}{2}(2\mu v + w)^2+\frac{1}{2} w^2 + 2\sigma^2 (1-\cos v)\Big]^{\frac{1}{2}} = \Big[2\mu^2 v^2 + 2 \mu vw + w^2 + 2\sigma^2 (1-\cos v)\Big]^{\frac{1}{2}}. 
        \]
        Then $V(0)=0$ and $V$ satisfies \eqref{Vbounds} with
		\begin{eqnarray}\label{Rpendulum}
        && (\frac{\mu}{\sqrt{2}} \wedge \frac{1}{\sqrt{3}})\|z\| \leq V(z) \leq \sqrt{3\mu^2 +\sigma^2+2} \|z\|; \notag\\
        && \langle \nabla V(z),f(z) \rangle = -\frac{\mu}{V(z)} (\sigma^2 v \sin v + w^2) \leq -\frac{\mu (\frac{\sigma^2}{2} \wedge 1) \|z\|^2}{V(z)} \leq -\frac{\mu (\frac{\sigma^2}{2} \wedge 1)}{3\mu^2 +\sigma^2+2}V(z)
		\end{eqnarray}
        for all $z\ne 0$ on the domain $\mathcal{D}:=B(0,1)$, where we use the inequalities
        \begin{eqnarray*}
       v\sin v \geq \frac{1}{2}v^2\quad \text{and}\quad 
1-\cos v = 2 \sin^2 (\frac{v}{2}) \leq \frac{1}{2}v^2,\quad  \forall |v| \leq 1. 
        \end{eqnarray*}
    In particular, for a special diffusion $g$ that $g(0)=0$, the trivial solution becomes an equilibrium. In that case, we can apply Theorem \ref{localstabLya} to prove that the trivial solution of \eqref{fSDE0} are still locally exponentially stable in a random sub-domain of $\mathcal{D}$. 
	\end{example}

     \begin{example}
         We revisit the example in \cite[Section 7]{gruene} by considering the system
         \begin{equation}
         \begin{split}
             \dot{y}_1&= -y_1 - 10y^2_2 \\
             \dot{y}_2 &= -2y_2. 
         \end{split}
         \end{equation}
        To investigate the stability of the trivial solution, the author in \cite{gruene} construct a Lyapunov function $V(y) = y_1^2 + y_2^2 + 13 y_2^4$. We will introduce a new strong Lyapunov function which satisfies condition \eqref{gradVmodified}. Indeed, construct the function $V(y) = (y_1^2+13 y_2^4)^{\frac{1}{4}}$. Then $V(y)\geq V(0)=0$, and by Cauchy inequality
        \begin{eqnarray*}
        \|y\| \leq \|y\|^{\frac{1}{2}}&\leq& V(y) \leq 13(1+\|y\|),\quad \forall \|y\| \leq 1;\\
        \|\nabla V(y)\| &= &\left\|\frac{(\frac{1}{2} y_1, 13 y_2^3)}{V(y)^3} \right\|\leq 1 + 13=14.    
        \end{eqnarray*}
        A direct computation shows that
        \begin{eqnarray*}
        \langle \nabla V(y), f(y)\rangle &=& \frac{-2 y_1^2 -20 y_1 y_2^2 -104 y_2^4}{4V^3(y)} =  -\frac{ y_1^2 + 4 y_2^4 + (y_1+10 y_2^2)^2}{4V^3(y)}\\
            &\leq& - \frac{1}{13}\frac{V(y)^4}{V(y)^3} 
            = -\frac{1}{13} V(y),\quad \forall y \ne 0. 
        \end{eqnarray*}
        Hence, $V$ satisfies \eqref{Vbounds}, \eqref{lipschitzV} and \eqref{gradVmodified} on the domain $\cD = B(0,1)$, and Theorem \ref{localstabLya} can be applied to prove the locally exponential stability of the trivial solution in $\cD$. 
     \end{example}   
     
		\section{Strong Lyapunov function approximation by neural networks}\label{Lyanetsec}
		As it is not an easy task to find an explicit strong Lyapunov function, we discuss in this section the problem of approximating our strong Lyapunov function on a compact domain $\mathcal{D}$ by a neural network. For simplicity, we will assume that a strong Lyapunov function $V$ exists and satisfies additional condition \eqref{Vbounds}. 
        
        Following \cite{gruene} and \cite{gabyetal} (see also \cite{richardsetal}, \cite{rodriguezetal}, \cite{zhouetal} and the references therein) we approximate $V$ by a Lyapunov neural network $V_\vartheta: \R^d \to \R$ from a single hidden layer neural network $N_\vartheta: \R^d \to \R$ with the set of its $\kappa$-trainable parameters $\vartheta \in \R^\kappa$,
		given by
		\begin{equation}\label{LyaNet}
			V_\vartheta (z) := |N_\vartheta(z)| + \bar{\alpha} \|z\|
		\end{equation}
		where $\bar{\alpha} >0$ is a small user-chosen parameters. Then $V_\vartheta$ satisfies conditions \eqref{Vbounds} and \eqref{lipschitzV}. As shown in \cite{gabyetal}, \cite{gruene}, the set of Lyapunov-Nets of the form \eqref{LyaNet} with a RePU network\footnote{this means \emph{rectified power units}, that is, $\sigma(x)=0$ if $x\le 0$ and $\sigma(x)=x^s$ for $x>0$, with $s$ a non-negative integer, for instance $s=2$; $s=1$ is the standard ReLU version} $N_\vartheta$ and bounded weights $\vartheta \in [-1,1]^\kappa$ is dense in the function space $\mathcal{S}:= \{\bar{h}(z) + \bar{\alpha} \|z\|: \bar{h} \in C^1(\cD,\R_+)\}$ under $W^{1,\infty}(\cD)$, where
            \[
			\|h\|_{W^{1,\infty}(\cD)}:= \max \{{\rm ess sup}_{z \in \cD} \|h(z)\|, \max \limits_{1 \leq i \leq d} {\rm ess sup}_{z \in \cD} \|D_ih(z)\|\}.
			\] 
        
        Similar to \cite{gabyetal}, we prove the following auxiliary result that there exists a strong Lyapunov function $V_{\vartheta^*}$ from the Lyapunov neural network.
		\begin{lemma}\label{lem3}
			For any $\varepsilon_0 \in (0,\alpha)$ and $\delta \in (0,1)$ in \eqref{gradientnegativity}, there exists $\vartheta^* \in \Theta$ such that $V_{\vartheta^*}$ of the form \eqref{LyaNet} satisfies \eqref{Vbounds} and 
			\begin{equation}\label{V*net}
					\sup \limits_{\substack{\psi \in \R^{d\times d},\|\psi\|\leq \lambda \\\eta \in \R^d,\|\eta\| \leq \lambda}}\langle\nabla V_{\vartheta^*}(z),(I + \psi) f(z+\eta)\rangle \leq  C_\lambda + \varepsilon_0 \Big[\delta+(1+\lambda)\| f\|_{\infty,\cD}\Big]-\delta V_{\vartheta^*}(z),\quad \forall z \in \cD.
			\end{equation}
        \end{lemma}	
		  \begin{proof}
			Observe that the set $\mathcal{V}^{\varepsilon_0}_\Theta$ of Lyapunov-Nets forms an $\varepsilon_0$-net of $\mathcal{S}$ in the $W^{1,\infty}(\cD)$ sense. Since $V \in \mathcal{S}$, there exists $\vartheta^* \in \Theta$ such that $V_{\vartheta^*} \in \mathcal{V}^{\varepsilon_0}_\Theta$ and $\|V - V_{\vartheta^*}\|_{W^{1,\infty}(\cD)} < \varepsilon_0$. It is easy to check that
            \begin{equation}\label{Vlip}
				\begin{split}
			V_{\vartheta^*}(z) &\leq V(z)+\|\nabla V_{\vartheta^*} -\nabla V^* \|_{\infty,\cD} \|z\| + |V(0)-V_{\vartheta^*}(0)| \leq (\beta+\varepsilon_0) \|z\|+ |V(0)-V_{\vartheta^*}(0)|; \\
			V_{\vartheta^*}(z) &\geq V^*(z) -\|\nabla V_{\vartheta^*} -\nabla V^* \|_{\infty,\cD} \|z\| - |V(0)-V_{\vartheta^*}(0)| \geq (\alpha - \varepsilon_0)\|z\|  - |V(0)-V_{\vartheta^*}(0)|;				
				\end{split}
			\end{equation}
            thus $V_{\vartheta^*}$ also satisfies \eqref{Vbounds} if we choose $\varepsilon_0 \in (0,\alpha)$. On the other hand, for any $\|\psi\|,\|\eta\|\leq \lambda$ 
             \allowdisplaybreaks
			\begin{eqnarray*}
			&&	\langle\nabla V_{\vartheta^*}(z), (I + \psi) f(z+\eta)\rangle \\
			&=& \langle\nabla V(z), (I + \psi) f(z+\eta)\rangle + \langle\nabla V_{\vartheta^*}(z) -\nabla V(z) ,(I + \psi) f(z+\eta)\rangle \\
			&\leq& C_\lambda-\delta V(z) + \|\nabla V_{\vartheta^*}(z) -\nabla V(z) \| \|I + \psi\|\| f(z+\eta)\| \\
			&\leq& C_\lambda-\delta V_{\vartheta^*}(z) + \delta \|V^*(z)- N_{\vartheta^*}(z)\| + \varepsilon_0 (1+\lambda)\| f\|_{\infty,\cD} \\
			&\leq& C_\lambda + \varepsilon_0 \Big[\delta+(1+\lambda)\| f\|_{\infty,\cD}\Big]-\delta V_{\vartheta^*}(z).
			\end{eqnarray*}
		which proves \eqref{V*net}.
		\end{proof}	
        
 Observe that $\|\nabla V_\vartheta\|, \|\nabla^2 V_\vartheta\| \leq L(\Theta,\cD)$ for a certain constant $L(\Theta,\cD)>0$ dependent on the set of bounded weights $\Theta$ and the compact domain $\cD$. We prove that.
    \begin{proposition}\label{Mvartheta}
        The function $M_\vartheta(z,\psi,\eta):= \langle\nabla V_\vartheta(z),(I + \psi) f(z+\eta)\rangle$ is continuous w.r.t. $(z,\psi,\eta)$ in the domain $ \cD \times B_{d\times d}(0,\lambda) \times B_d(0,\lambda)$. In particular,
         \begin{equation}\label{Mest}
         |M(\tz,\tpsi,\teta)-M(z,\psi,\eta)|  \leq L(\Theta,\cD)(1+\lambda) \Big(\|f\|_{\infty,B(\cD,\lambda)} +\|Df\|_{\infty,B(\cD,\lambda)} \Big) \Big(\|\tpsi-\psi\|+\|\teta-\eta\|+\|\tz-z\|\Big).
         \end{equation}
    \end{proposition}
            \begin{proof}
       The proof follows directly from the estimate
         \allowdisplaybreaks
        \begin{eqnarray*}
            &&|M(\tz,\tpsi,\teta)-M(z,\psi,\eta)| \notag\\
            &\leq& |M(\tz,\tpsi,\teta)-M(\tz,\psi,\eta)| + |M(\tz,\psi,\eta)-M(z,\psi,\eta)| \notag\\
            &\leq& \|\nabla V_\vartheta(\tz)\|\Big( \|(I+\tpsi)f(\tz+\teta)-(I +\psi)f(\tz+\teta)\| +\|(I+\psi)\|\|[f(\tz+\teta)-f(\tz+\eta)]\|\Big)\notag\\
            &&+ |\langle\nabla V_\vartheta(\tz)-\nabla V_\vartheta(z),(I + \psi) f(\tz+\eta)\rangle| + |\langle\nabla V_\vartheta(z),(I + \psi) [f(\tz+\eta)-f(z+\eta)]\rangle| \notag\\
            &\leq& L(\Theta,\cD) \Big(\|\tpsi-\psi\| \|f\|_{\infty,B(\cD,\lambda)} + (1+\lambda) \|Df\|_{\infty,B(\cD,\lambda)}\|\teta-\eta\|\Big) \notag\\
            &&+  L(\Theta,\cD) \|\tz-z\| (1+\lambda) \|f\|_{\infty,B(\cD,\lambda)} +L(\Theta,\cD)(1+\lambda) \|Df\|_{\infty,B(\cD,\lambda)} \|\tz-z\| \notag\\
            &\leq& L(\Theta,\cD)(1+\lambda) \Big(\|f\|_{\infty,B(\cD,\lambda)} +\|Df\|_{\infty,B(\cD,\lambda)} \Big) \Big(\|\tpsi-\psi\|+\|\teta-\eta\|+\|\tz-z\|\Big).
        \end{eqnarray*}
    \end{proof}
  
As a result, for any fixed parameter $\epsilon \in (0,1)$, one can apply \eqref{V*net} and \eqref{Mest} to choose $\epsilon_0 \in (0,1)$ small enough such that
    \begin{equation}\label{eps0}
   \varepsilon_0  \max\Big\{ \Big[\delta+(1+\lambda)\| f\|_{\infty,\cD}\Big], 3L(\Theta,\cD)(1+\lambda) \Big(1+\|f\|_{\infty,B(\cD,\lambda)} +\|Df\|_{\infty,B(\cD,\lambda)} \Big) \Big\} \leq \epsilon.
    \end{equation}
        The training process of Lyapunov neural network $V_\vartheta$ in \eqref{LyaNet} aims to find a specific network parameter $\vartheta$ such that the condition \eqref{gradientnegativity} holds for $V_\vartheta$ at every $z \in \cD$. This can be achieved by minimize the risk function
		\begin{equation}\label{riskfunc}
			l(\vartheta) := \frac{1}{|\cD|} \int_\cD \Big[\sup \limits_{\|\psi\|\leq \lambda, \|\eta\| \leq \lambda} \langle \nabla V_\vartheta(z),(I + \psi) f(z+\eta(z))\rangle + \bar{\delta}V_\vartheta(z)-\bar{C_\lambda}\Big]^2_+ dz  
		\end{equation}
		for a user chosen parameters $\bar{\delta} \in (0,1)$ and $\bar{C_\lambda} >0$. Note that due to \eqref{V*net} and \eqref{eps0}, $\bar{C_\lambda}$ should theoretically be chosen to be $\epsilon$-close to $C_\lambda$. 
        
        In practice, the integral \eqref{riskfunc} does not have analytic form, thus we approximate $l(\vartheta)$ using the empirical expectation  
		\begin{equation}\label{empriskfunc}
			\hat{l}(\vartheta) := \frac{1}{m_z(\cD,\epsilon_0,\rho)} \sum_{i =1}^{m_z(\cD,\epsilon_0,\rho)} \Big[\sup \limits_{\psi_j, \eta_k} \langle \nabla V_\vartheta(z_i), (I + \psi_j) f(z_i+\eta_k)\rangle + \bar{\delta} V_\vartheta(z_i)-\bar{C_\lambda}\Big]^2_+   
		\end{equation}
		where $\{z_i: i = 1,\ldots,m_z(\cD,\epsilon_0,\rho)\}$ (respectively $\{\psi_j: j = 1,\ldots,m_\psi(B_{d\times d}(0,\lambda),\epsilon_0,\rho)\}$ and $\{\eta_k: k = 1,\ldots,m_\eta(B_d(0,\lambda),\epsilon_0,\rho)\}$) are independent and identically distributed (i.i.d) samples from the uniform distribution $U(\cD)$ on $\cD$ (respectively $U(B_{d\times d}(0,\lambda))$ on $B_{d\times d}(0,\lambda)$ and $U(B_d(0,\lambda))$ on $B_d(0,\lambda)$). As a result, the minimizer of $\hat{l}$ exists and achieves minimal value. Given a parameter $\rho \ll 1$ small enough and consider the samplings in Appendix \ref{app2}, we will show below that.
        \begin{theorem}\label{accLya}
        With a probabilty $1-\rho$, any minimizer $\hat{\vartheta}$ of $\hat{l}(\cdot)$ generates $V_{\hat{\vartheta}}$ that is an $\epsilon$-accurate strong Lyapunov function of $f$ on $\cD$.
        \end{theorem}
\begin{proof}
	Due to Lemma \ref{lem3}, $\hat{l}(\vartheta^*) =0$ for an appropriate choice of $\bar{\delta}, \bar{C_\lambda}$ and due to the non-negativeness of $\hat{l}$, the minimum of $\hat{l}$ on $\Theta$ is $0$. Let $\hat{\vartheta}$ be any minimizer of $\hat{l}$ then $\hat{l}(\hat{\vartheta}) =0$, which implies that
	\begin{equation}\label{Vhat}
	\sup \limits_{\psi_j, \eta_k}	M_{\hat{\vartheta}}(z_i,\psi_j,\eta_k) \leq \bar{C_\lambda}- \bar{\delta} V_{\hat{\vartheta}}(z_i),\quad \forall i = 1,\ldots, m_z(\cD,\epsilon_0,\rho).
	\end{equation}
	Next given a probability $1-\rho$, for any $z \in \cD, \psi \in B_{d\times d}(0,\lambda), \eta \in B_d(0,\lambda)$, there exists $z_i, \psi_j, \eta_k$ such that $z \in B(z_i,\epsilon_0), \psi \in B(\psi_j,\lambda), \eta \in B(\eta_k,\lambda)$. Since $\bar{\delta} <1$, \eqref{V*net} and \eqref{Mest} yield
	\begin{eqnarray}\label{Mest2}
		M_{\hat{\vartheta}}(z,\psi,\eta) &\leq& M_{\hat{\vartheta}}(z_i,\psi_j,\eta_k) + \epsilon \leq \bar{C_\lambda} + \epsilon - \bar{\delta} V_{\hat{\vartheta}}(z_i)\notag \\
        &\leq& \bar{C_\lambda} + \epsilon + \bar{\delta} |V_{\hat{\vartheta}}(z)-V_{\hat{\vartheta}}(z_i)| - \bar{\delta} V_{\hat{\vartheta}}(z)\notag\\
        &\leq& \bar{C_\lambda} + \epsilon + L(\Theta,\cD) \epsilon_0 - \bar{\delta} V_{\hat{\vartheta}}(z) \notag\\ 
        &\leq& \bar{C_\lambda} + 2\epsilon- \bar{\delta} V_{\hat{\vartheta}}(z) . 
	\end{eqnarray}
Taking the supremum of the left hand side of \eqref{Mest2} for $\psi\in B(\psi_j,\lambda), \eta \in B(\eta_k,\lambda)$, we conclude that $V_{\hat{\vartheta}}$ is a $\epsilon$-accurate Lyapunov function.
	\end{proof}

    \section*{Acknowledgments}
    We thank Phan Thanh Hong, Giacomo Landi, Christian Kühn for their careful proofreading and also for useful suggestions which led to improvements in the presentation of the paper.
		
		\section{Appendix}
		\subsection{Rough paths and the random setting}\label{roughpath}
		Let us briefly present the concept of rough paths in the simplest form, following   Friz  \& Hairer  \cite{frizhairer} and  Lyons \cite{lyons98}. For any finite dimensional vector space $W$, denote by $C([a,b],W)$ the space of all continuous paths $y: [a,b] \to W$ equipped with the sup norm $\|\cdot\|_{\infty,[a,b]}$ given by $\|y\|_{\infty,[a,b]}=\sup_{t\in [a,b]} \|y_t\|$, 
		where $\|\cdot\|$ is the norm in $W$. We write $y_{s,t}:= y_t-y_s$. For $p\geq 1$, denote by $C^{p{\rm-var}}([a,b],W)\subset C([a,b],W)$ the space of all continuous paths $y:[a,b] \to W$ of finite $p$-variation 
		\[
		\ltn y\rtn_{\tp,[a,b]} :=\left(\sup_{\Pi([a,b])}\sum_{i=1}^n \|y_{t_i,t_{i+1}}\|^p\right)^{1/p} < \infty, 
		\]
		where the supremum is taken over the whole class of finite partitions of $[a,b]$. 
		Also for each $0<\alpha<1$, we denote by $C^{\alpha}([a,b],W)$ the space of H\"older continuous functions with exponent $\alpha$ on $[a,b]$ equipped with the norm
		\begin{equation}\label{holnorm}
			\|y\|_{\alpha,[a,b]}: = \|y_a\| + \ltn y\rtn_{\alpha,[a,b]},\quad \text{where} \quad \ltn y\rtn_{\alpha,[a,b]} :=\sup_{\substack{s,t\in [a,b],\ s<t}}\frac{\|y_{s,t}\|}{(t-s)^\alpha} < \infty.
		\end{equation}
		Let  $\alpha \in (\frac{1}{3},\frac{1}{2})$ (to avoid integrals of  order higher than 2 in the theory to be described) and $x \in C^\alpha([a,b],\R^m)$. To arrive at a solution of a stochastic differential equation along $x$, it is necessary to supplement $x$ by a higher order integral. Thus, a couple $\bx=(x,\X) \in \R^m \oplus (\R^m \otimes \R^m)$, where
		\[
		\X \in C^{2\alpha}([a,b]^2,\R^m \otimes  \R^m):= \left\{\X \in C([a,b]^2,\R^m \otimes  \R^m):  \sup_{\substack{s, t \in [a,b],\ s<t}} \frac{\|\X_{s,t}\|}{|t-s|^{2\alpha}} < \infty \right\}, 
		\]
		is called a {\it rough path} if it satisfies Chen's relation
		\begin{equation}\label{chen}
			\X_{s,t} - \X_{s,u} - \X_{u,t} = x_{s,u} \otimes  x_{u,t},\qquad \forall a \leq s \leq u \leq t \leq b. 
		\end{equation}
        The two-parameter function $\X$ then postulates the values for the iterated integral
        \begin{equation}
            \X_{s,t}=\int_s^t x_{s,r}dx_r.
        \end{equation}
        Such integrals are needed for representing pathwise solutions of stochastic differential equations.
		We introduce the rough path semi-norm 
		\begin{equation}\label{translated}
			\ltn \bx \rtn_{\alpha,[a,b]} := \ltn x \rtn_{\alpha,[a,b]} + \ltn \X \rtn_{2\alpha,[a,b]^2}^{\frac{1}{2}},\quad \text{where}\quad \ltn \X \rtn_{2\alpha,[a,b]^2}:= \sup_{s, t \in [a,b];s<t} \frac{\|\X_{s,t}\|}{|t-s|^{2\alpha}} < \infty.  
		\end{equation}
		Throughout this paper, we will fix parameters $\frac{1}{3}< \alpha < \nu <\frac{1}{2}$ and  $p = \frac{1}{\alpha}$ so that $C^\alpha([a,b],W) \subset C^{\tp}([a,b],W)$. We also set $q=\frac{p}{2}$ (so that $\frac{1}{p} +\frac{1}{q}>1$, an important technical condition for ensuring the existence of rough integrals) and consider the $\tp$ semi-norm 
		\begin{equation}\label{pvarnorm}
			\begin{split}
				\ltn \bx \rtn_{\tp,[a,b]} := \Big(\ltn x \rtn^p_{\tp,[a,b]} + \ltn \X \rtn_{\tq,[a,b]^2}^q\Big)^{\frac{1}{p}}, \ltn \X \rtn_{\tq,[a,b]^2} := \left(\sup_{\Pi([a,b])}\sum_{i=1}^n \|\X_{t_i,t_{i+1}}\|^q\right)^{1/q}, 
			\end{split}
		\end{equation}
		where the supremum is taken over the whole class of finite partitions $\Pi([a,b])$ of $[a,b]$.  

        In the rough path setting, denote by $T^2_1(\R^m) = 1 \oplus \R^m \oplus (\R^m \otimes \R^m)$ the set with the group product
	\[
	(1,g^1,g^2) \bullet (1,h^1,h^2) = (1, g^1 + h^1, g^1 \otimes h^1 + g^2 +h^2),
	\]
	for all ${\bf g} =(1,g^1,g^2), {\bf h} = (1,h^1,h^2) \in T^2_1(\R^m)$.	Given $\alpha \in (\frac{1}{3},\nu)$, denote by $\cC^{0,\alpha}(I,T^2_1(\R^m))$ 
	the closure of $\cC^{\infty}(I,T^2_1(\R^m))$ in the H\"older space $\cC^{\alpha}(I,T^2_1(\R^m))$, and by 
	$\cC_0^{0,\alpha}(\R,T^2_1(\R^m))$ the space of all paths $ {\bf g}: \R\to T^2_1(\R^m))$ such that $ {\bf g}|_I \in \cC^{0,\alpha}(I, T^2_1(\R^m))$ for each compact interval $I\subset\R$ containing $0$. Assign $\Omega:=\cC_0^{0,\alpha}(\R,T^2_1(\R^m))$ and equip it with the Borel $\sigma$-algebra $\mathcal{F}$. Denote by $\theta$ the {\it Wiener-type shift}
	\begin{equation}\label{shift}
		(\theta_t \omega)_\cdot = \omega_t^{-1}\bullet \omega_{t+\cdot},\forall t\in \R, \omega \in \cC^{0,\alpha}_0(\R,T^2_1(\R^m)),
	\end{equation}  
	and define the so-called {\it diagonal process}  $\bX: \R \times \Omega \to T^2_1(\R^m), \bX_t(\omega) = \omega_t$ for all $t\in \R, \omega \in \Omega$. Under assumption 	(${\textbf H}_X$), it can be proved that there exists a probability measure $\bP$ which is $\theta$ - invariant \cite[Theorem 5]{BRSch17}. Thus $(\Omega,\mathcal{F},\mP)$ is a probability space equipped with the continuous (thus measurable) {\it metric dynamical system} $\theta: \R \times \Omega \to \Omega$. 
	In particular, the Wiener shift \eqref{shift} implies that
	\begin{equation}\label{roughshift}
		\ltn \bx(\theta_h \omega) \rtn_{\tp,[s,t]} = \ltn \bx(\omega) \rtn_{\tp,[s+h,t+h]}.
	\end{equation}
	It is proved in \cite[Lemma 6.1]{duchongcong24} that $\theta$ is ergodic if $X=B^H$ is a fractional Brownian motion. In this paper, we assume that the metric dynamical system $\theta$ is ergodic. 

    Note that when dealing with additive noise, we do not need rough path lifts but instead consider $\Omega:=\cC_0^{0,\alpha}(\R,\R^m)$ together with a Wiener shift $(\theta_t \omega)_\cdot = \omega_{t+\cdot} -\omega_t$.

	   \subsection{Generating i.i.d. samples from a compact domain}\label{app2}
    Consider a compact domain $\cD$ and fix a parameter $\epsilon_0 \in (0,1)$. Due to compactness, $\cD$ is covered by $M(\epsilon_0,\cD)$ balls $B(y_i,\epsilon_0)$ where $y_i \in \cD$. Consider a sample of $m$ i.i.d. random variables $z_j$ from the uniform distribution $U(\cD)$ on $\cD$ and assign $p_i := \mP (z \in B(y_i,\epsilon_0)), i = 1,\ldots, M(\epsilon_0,\cD)$ and $p_{\rm min} := \min \{p_i: i = 1,\ldots, M(\epsilon_0,\cD)\}$. Then for any $\rho \ll 1$ small enough, we obtain (due to the i.i.d. samplings) the estimate
    \allowdisplaybreaks
    \begin{eqnarray*}
    &&\mP(\forall i = 1,\ldots,M(\epsilon_0,\cD), \exists j: z_j \in B(y_i,\epsilon_0)) \\
    &=& 1-  \mP(\exists i = 1,\ldots,M(\epsilon_0,\cD): z_j \notin B(y_i,\epsilon_0)\  \forall  j = 1,\ldots, m) \\
   &\geq& 1 - \sum_{i = 1}^{M(\epsilon_0,\cD)} \mP(z_j \notin B(y_i,\epsilon_0)\  \forall  j = 1,\ldots, m)\\
   &\geq& 1 - \sum_{i = 1}^{M(\epsilon_0,\cD)} (1-p_i)^m
   \geq 1 - M(\epsilon_0,\cD) (1-p_{\rm min})^m \geq 1-\rho   
    \end{eqnarray*}
    if we choose $m = m(\cD,\epsilon_0,\rho)$ large enough. Therefore for any $\rho \ll 1$, there exists a sampling size $m = m(\rho,\epsilon_0)$ large enough such that for any sampling of i.i.d. $\{z_j: i = 1,\ldots,m\}$ from the uniform distribution $U(\cD)$ we obtains $\cD \subset \bigcup_{j =1}^{m} B(z_j,\epsilon_0)$ with a probability larger than $1-\rho$. 
    
    \subsection{Proofs}\label{allproofs}
\begin{proof}[{\bf Lemma \ref{Kfunctions}}]
    The second statement comes obviously from Young's inequality, thus we only prove the first statement. Since $\beta$ satisfies \eqref{tempered}, for any $\epsilon_0 >0$, there exists an $\epsilon(\beta,\epsilon_0) >0$ such that: for any $\epsilon_1 < \epsilon(\beta,\epsilon_0)$ there exists a $t_0(\beta,\epsilon_1,\epsilon_0)>0$ such that
    \begin{equation}\label{tempered1}
    \beta (e^{\epsilon_1 t}) < e^{\epsilon_0 t},\quad \forall t\geq  t_0(\beta,\epsilon_1,\epsilon_0).
    \end{equation}
Now since $\alpha$ also satisfies \eqref{tempered}, there exists with the above $\epsilon_1$ an $\epsilon(\alpha,\epsilon_1) >0$ such that for any $\epsilon_2 < \epsilon(\alpha,\epsilon_1)$ there exists a $t_1(\alpha,\epsilon_2,\epsilon_1)>0$ such that
    \begin{equation}\label{tempered2}
    \alpha (e^{\epsilon_2 t}) < e^{\epsilon_1 t},\quad \forall t\geq  t_1(\alpha,\epsilon_2,\epsilon_1).
    \end{equation} 
    Combining \eqref{tempered1} and \eqref{tempered2} and the fact that $\beta,\alpha$ are strictly increasing functions in $\R_+$, we obtain 
    \begin{equation}\label{tempered3}
        \beta (\alpha (e^{\epsilon_2 t})) < \beta (e^{\epsilon_1 t}) < e^{\epsilon_0 t},\quad \forall t \geq  \max \{t_0(\beta,\epsilon_1,\epsilon_0), t_1(\alpha,\epsilon_2,\epsilon_1)\}.
    \end{equation}
  Since \eqref{tempered3} holds for any $\epsilon_0 >0$, it shows that $\beta\circ \alpha$ also satisfies \eqref{tempered}.  \\
\end{proof}

\begin{proof}[{\bf Lemma \ref{thetaL1}}]
i, Assume $|\xi(\omega)|$ is tempered w.r.t. $\omega$, i.e. $\limsup \limits_{t \to \infty} \frac{1}{t} \log |\xi(\theta_{-t}\omega)| =0$, it follows that for any $\epsilon \in (0,\delta)$, there exists $T(\omega,\epsilon)$ such that $|\xi(\theta_{-t}\omega)| \leq e^{\epsilon t}$ for all $t\geq T(\omega,\epsilon)$. Thus 
\begin{eqnarray*}
\bar{\xi}(\omega):=\sup \limits_{t\geq 0} e^{-\delta t} |\xi(\theta_{-t}\omega)| &\leq& \max \{\sup \limits_{t < T(\omega,\epsilon)} e^{-\delta t} |\xi(\theta_{-t}\omega)|, \sup \limits_{t \geq T(\omega,\epsilon)}e^{(\epsilon -\delta)t}\} \\
&\leq&  \max \{\sup \limits_{t < T(\omega,\epsilon)} e^{-\delta t} |\xi(\theta_{-t}\omega)|, e^{(\epsilon -\delta)T(\omega,\epsilon)}\}; 
\end{eqnarray*}
which proves that $\bar{\xi}(\omega)$ is well defined and finite. Moreover, for any $s\geq T(\omega,\epsilon)$
\begin{eqnarray*}
\bar{\xi}(\theta_{-s}\omega)=e^{\delta s}\sup \limits_{t\geq 0} e^{-\delta (t+s)} |\xi(\theta_{-t-s}\omega)|\leq e^{\delta s}\sup \limits_{t\geq s} e^{-\delta t} |\xi(\theta_{-t}\omega)| 
\leq e^{\delta s}\sup \limits_{t\geq s} e^{(\epsilon-\delta)t}  = e^{\delta s}e^{(\epsilon-\delta)s} = e^{\epsilon s}.
\end{eqnarray*}
This proves that $\bar{\xi}(\omega)$ is also tempered w.r.t. $\omega$.

   ii, The proof is straight forward from the measure preseration of $\theta$ and from Fubini's theorem. Specifically, since $|\xi|$ is integrable, it is tempered a.s. (see e.g. \cite[Lemma 3.4.3]{arnold}), thus $\sup \limits_{t\in \R_+} e^{-\delta t} |\xi(\theta_{-t}\omega)|$ is well defined and finite a.s. 
    For any $m \in \N_+$ 
       \allowdisplaybreaks
    \begin{eqnarray*}
       && \sum_{n=m}^\infty (n+1) \mP(\omega: \sup_{t\in \R_+} e^{-\delta t} |\xi(\theta_{-t}\cdot)| \in [n,n+1]) \\
       &\leq& \sum_{n=m}^\infty (n+1) \int_0^\infty \mP(\omega: e^{-\delta t} |\xi(\theta_{-t}\cdot)| \in [n,n+1])dt \\
       &\leq& \sum_{n=m}^\infty (n+1) \int_0^\infty \mP(\omega: e^{-\delta t} |\xi(\cdot)| \in [n,n+1])dt \\
       &\leq& \sum_{n=m}^\infty (n+1) \int_0^\infty \mP(\omega: |\xi(\cdot)| \in [e^{\delta t} n,e^{\delta t} (n+1)])dt \\
       &\leq& \int_0^\infty e^{-\delta t}\Big[\sum_{n=m}^\infty e^{\delta t}(n+1) \mP(\omega: |\xi(\cdot)| \in [e^{\delta t} n,e^{\delta t} (n+1)])\Big]dt \\
        &\leq& \int_0^\infty e^{-\delta t}\Big[\sum_{n=m}^\infty \big(\frac{n+1}{n}\big) e^{\delta t}n\mP(\omega: |\xi(\cdot)| \in [e^{\delta t} n,e^{\delta t} (n+1)])\Big]dt \\
          &\leq& \int_0^\infty e^{-\delta t}\big(\frac{m+1}{m}\big)\Big[\sum_{n=m}^\infty e^{\delta t}n\mP(\omega: |\xi(\cdot)| \in [e^{\delta t} n,e^{\delta t} (n+1)])\Big]dt \\
       &\leq& \E |\xi(\cdot)|\big(\frac{m+1}{m}\big) \int_0^\infty   e^{-\delta t}dt =\frac{1}{\delta}\big(\frac{m+1}{m}\big)\E |\xi(\cdot)|.
    \end{eqnarray*}
    This proves that $\sum_{n=0}^\infty (n+1) \mP(\omega: \sup \limits_{t\in \R_+} e^{-\delta t} |\xi(\theta_{-t}\cdot)| \in [n,n+1]) <\infty$ and thus leading to $\sup \limits_{t\in \R_+} e^{-\delta t} |\xi(\theta_{-t}\cdot)| \in \cL^1$.\\
\end{proof}
      
\begin{proof}[{\bf Lemma \ref{Rlem}}]
    Observe from \eqref{conditionx} in (${\textbf H}_X$) that $H(\ltn \bx(\omega) \rtn_{\tp,[-1,0]}) \in \cL^\rho$ for any $\rho \geq 1$, thus $H$ is also tempered. Hence by applying \cite[Lemma 5.2]{congduchong22} and \cite[Lemma 4.1]{congduchong23} one can easily check that $\bar{R}_\lambda (\cdot)$ is finite and in $\cL^\rho$, thus is also tempered a.s. To prove \eqref{Radiusest}, observe from \eqref{H} and Young's inequality that
    \begin{equation}\label{H2}
		H(\xi) < C_H (\xi^p+ 1).
	\end{equation}	    
 Assign $n(T):= \min \{n\in \N_+: n \geq T\}$. Applying the Wiener shift property \eqref{roughshift}, then \eqref{H2}, and \cite[Lemma 2.1]{congduchong17}, we obtain
\begin{eqnarray*}
 \max \limits_{t\in [0,T]} \bar{R}_\lambda(\theta_{-t}\omega) &\leq& \sum_{k=0}^\infty e^{-k\delta} \max \limits_{t\in [0,T]} H(\ltn \bx(\theta_{-t}\omega) \rtn_{\tp,[-1-k,-k]})\\
 &\leq& \sum_{k=0}^\infty e^{-k\delta} \max \limits_{t\in [0,T]} H(\ltn \bx(\omega) \rtn_{\tp,[-t-1-k,-k-t]})\\
 &\leq& \sum_{k=0}^\infty e^{-k\delta} H(\ltn \bx(\omega) \rtn_{\tp,[-T-1-k,-k]})  \\
 &\leq& C_H\sum_{k=0}^\infty e^{-k\delta} (1+\ltn \bx(\omega) \rtn^p_{\tp,[-n(T)-1-k,-k]})\\
 &\leq& C_H n(T)^{p-1} \sum_{k=0}^\infty e^{-k\delta} \Big(1+\sum_{i=0}^{n(T)}\ltn \bx(\omega) \rtn^p_{\tp,[-i-1-k,-i-k]}\Big)\\
&\leq& C_H n(T)^p \sum_{k=0}^\infty e^{-k\delta} \Big(1+\ltn \bx(\omega) \rtn^p_{\tp,[-1-k,-k]}\Big)<\infty \quad \text{a.s.}
\end{eqnarray*}
where the last line is due to the fact that $\ltn \bx(\cdot) \rtn^p_{\tp,[-1,0]} \in \cL^\rho$ for any $\rho \geq 1$.  \\  
\end{proof}

\begin{proof}[{\bf Estimate \eqref{hestdiscrete}}] We sketch the proof here, since the arguments go line by line with the proof of \cite[Theorem 3.2 \& Proposition 3]{duc21}, but are adapted for the discrete time setting, by the usage of the discrete sewing lemma \cite[Lemma 2.2]{congduchong23} and the discrete Gronwall lemma. First, let 
\begin{eqnarray*}
y^\Delta_{t_k} &:=& \varphi(t_k,\omega)y_0^\Delta,\quad  \mu^\Delta_{t_k}:= \Phi^\Delta(t_k)\mu_0^\Delta, \quad \mu_{t_k} := \Phi^\Delta(t_k)y_0^\Delta,\\
y^*_{t_k} &:=&  y^\Delta_{t_k}-\mu^\Delta_{t_k},\qquad \mu^*_{t_k} := \mu_{t_k} - \mu^\Delta_{t_k},    
\end{eqnarray*}
for any $t_k \in \Pi[0,T]$ for $\mu_0^\Delta \in \cA^{\Delta,0}$. Due to estimates \eqref{Phideltaest}, 
\begin{equation}\label{mustar}
\begin{split}
\ltn \mu^\Delta \rtn_{1-{\rm var},\Pi[s,t]} & \leq (t-s)\|f\|_{\infty,\cA^{\Delta,0}};\quad
\|\mu^*\|_{\infty,\Pi[0,T]} \leq e^{C_fT} \|\mu^*_0\|;\\
\ltn \mu^* \rtn_{1-{\rm var},\Pi[s,t]} &\leq (t-s)e^{C_fT}\|\mu^*_0\|.   
\end{split}
\end{equation}
Assign $h_{t_k} := y^*_{t_k}-\mu^*_{t_k} = y^\Delta_{t_k} - \mu_{t_k}$ for all $t_k\in \Pi[0,T]$, then $y^\Delta_{t_k} =  \mu^\Delta_{t_k} + \mu^*_{t_k} +h_{t_k}$ and
\begin{eqnarray}\label{hequdis}
h_{t_{k+1}} &= &h_{t_k}+[g(y^\Delta_{t_k}) x_{t_k,t_{k+1}} + D g(y^\Delta_{t_k}) g(y^\Delta_{t_k})\X(t_k,t_{k+1})]+ [f(y^\Delta_{t_k}) -f(\mu_{t_k})  ]  (t_{k+1}-t_k) \notag\\
&=:& h_{t_k}+ F_{t_k,t_{k+1}} +[f(y^\Delta_{t_k}) -f(\mu_{t_k})  ]  (t_{k+1}-t_k)
\end{eqnarray}
in which $F_{s,t}: = g(y^\Delta_s)x_{s,t} + Dg(y^\Delta_s)g(y^\Delta_s)\X_{s,t}$. We define an auxiliary quantity $R^h_{s,t}:= h_{s,t} - g(y^\Delta_s)x_{s,t} = R^{y^\Delta}_{s,t} - \mu^\Delta_{s,t}$. Then
\[
\|R^{y^\Delta}_{s,t} \| \leq \|R^h_{s,t} \| +\|\mu^\Delta_{s,t}\|.
\]
Similar to \cite[Proposition 3]{duc21}, we can prove that 
\begin{eqnarray*}
&&\ltn g(\mu^\Delta + \mu^* +h) \rtn_{\tp,\Pi[s,t]} \vee \ltn Dg(\mu^\Delta + \mu^* +h) \rtn_{\tp,\Pi[s,t]} \\
&\leq& C_g \ltn h \rtn_{\tp,\Pi[s,t]} + C_g \ltn \mu^\Delta \rtn_{1-{\rm var},\Pi[s,t]} + 2C_g \ltn \mu^* \rtn^{\frac{2}{p}}_{1-{\rm var},\Pi[s,t]},\quad \forall 0\leq s\leq t\leq T.   
\end{eqnarray*}
On the other hand, by applying the discrete sewing lemma to $F_{s,t}$, we obtain from \eqref{hequdis} that
      \begin{eqnarray}\label{h}
 	   &&  \ltn h,R^h\rtn_{\tp,\Pi[s,t]} = \ltn h\rtn_{\tp,\Pi[s,t]}+\ltn R^h\rtn_{\tq,\Pi[s,t]} \notag\\
    &\leq&2 C_f \Delta\sum_{s\leq t_k < t} \|h_{t_k}\|+4C_p \Big(C_g\ltn \bx\rtn_{\tp,\Pi[s,t]} \vee C_g^2\ltn \bx \rtn^2_{\tp,\Pi[s,t]} \Big) \ltn h,R^h\rtn_{\tp,\Pi[s,t]}\notag\\
      && + (C_g\ltn \bx\rtn_{\tp,\Pi[s,t]}  +2C_g^2\ltn \bx\rtn^2_{\tp,\Pi[s,t]} + C_p C_g^3 \ltn \bx\rtn^3_{\tq,\Pi[s,t]}) \times \notag\\
      && \hspace{1cm}\times \left[1+ 4C_p\ltn \mu^*\rtn^{2/p}_{1{\rm-var},\Pi[s,t]} + 4C_p\ltn \mu^\Delta\rtn_{1{\rm-var},\Pi[s,t]}  \right] \\
      &\leq&2 C_f \Delta\sum_{s\leq t_k < t} \|h_{t_k}\|+4C_p \Big(C_g\ltn \bx\rtn_{\tp,\Pi[s,t]} \vee C_g^2\ltn \bx \rtn^2_{\tp,\Pi[s,t]} \Big) \ltn h,R^h\rtn_{\tp,\Pi[s,t]} + L\notag
     \end{eqnarray}
 where the third term $L$ in \eqref{h} can be written in the form  
 \begin{equation}\label{Lest}
 L:=C_g\xi(C_f,T,\ltn \bx\rtn_{\tp,[0,T]},\|f\|_{\infty,\cA^{\Delta,0}}) (1+\|\mu^*_0\|^{\frac{2}{p}})
 \end{equation}
 for a (generic) integrable random variable $\xi$, due to \eqref{mustar}.
 
Now we repeat arguments in \cite[Theorem\ 3.3]{congduchong23}, to construct the same sequence of stopping times $\{\tau^\Delta_i\}$ on $\Pi[0,T]$ as follows. Assign $\tau^\Delta_0 = 0$. For each $n\in \N$, assume $\tau^\Delta_n =t_l$ is determined, one can define $\tau^\Delta_{n+1}$ by the following rule:
        \begin{itemize}
			\item if $4C_p C_g \ltn \bx \rtn_{\tp, \Pi[t_l,t_{l+1}]} \leq \frac{1}{2}$ then set $\tau^\Delta_{n+1}:= \sup \{t_m >t_l:  4C_p C_g \ltn \bx \rtn_{\tp, \Pi[t_l,t_m]} \leq\frac{1}{2} \}$;
			\item else set $\tau^\Delta_{n+1}:= t_{l+1}$. 	
		\end{itemize}   
Let consider two cases from \eqref{h} according to the two cases of the discrete stopping times. In the first case, it follows from \eqref{h} that
\[
\|h_t\| \leq \|h_s\|+\ltn h \rtn_{\tp,\Pi[s,t]} \leq \|h_s\|+4 C_f\Delta \sum_{s \leq t_k < t} \|h_{t_k}\| + 2L,\quad \forall s,t \in \Pi,\ \tau^\Delta_n \leq s\leq t \leq \tau^\Delta_{n+1}.
\]
Hence by induction, we can prove that 
\[
\|h_t\| \leq \|h_s\| (1+4C_f\Delta)^{\frac{t-s}{\Delta}} + 2L  (1+4C_f\Delta)^{\frac{t-s}{\Delta}-1},\quad \forall s, t\in \Pi, \  \tau^\Delta_n \leq s\leq t \leq \tau^\Delta_{n+1};
\]
which yields
\begin{equation}\label{hstoptime}
|h_{\tau^\Delta_{n+1}}\|\leq \|h\|_{\infty,\Pi[\tau^\Delta_n,\tau^\Delta_{n+1}]} \leq (\|h_{\tau^\Delta_n}\| + 2L) e^{4C_f(\tau^\Delta_{n+1}-\tau^\Delta_n)}.    
\end{equation}
In the second case, it follows from \eqref{hequdis} that
\begin{eqnarray*}
|h_{\tau^\Delta_{n+1}}\|=\|h_{t_{l+1}}\| &\leq& \|h_{t_l}\| + C_g \ltn \bx \rtn_{\tp,\Pi[t_l,t_{l+1}]} + C_g^2 \ltn \bx \rtn^2_{\tp,\Pi[t_l,t_{l+1}]} + 2C_f \Delta \|h_{t_l}\| \\
&\leq&\|h_{t_l}\| (1+4C_f\Delta) +2L\\
&\leq& (|h_{\tau^\Delta_n}\|+ 2L)e^{4C_f(\tau^\Delta_{n+1}-\tau^\Delta_n)}.
\end{eqnarray*}
Since \eqref{hstoptime} holds in both cases, we conclude from induction that 
\begin{equation}\label{hsupest}
\begin{split}
\|h\|_{\infty,\Pi[0,T]} &\leq \Big[\|h_0\| + 2L \hat{N}\Big(\frac{1}{2},\bx,[0,T]\Big)\Big] \exp \{4C_f(\tau^\Delta_{\hat{N}} -\tau^\Delta_0)\} \\
&\leq \Big[\|h_0\| + 2L \hat{N}\Big(\frac{1}{2},\bx,[0,T]\Big)\Big] e^{4C_fT},
\end{split}
\end{equation}
where $\hat{N}(\frac{1}{2},\bx,[0,T])$ is the number of stopping times $\tau^\Delta$ in the interval $[0,T]$. As proved in \cite[Lemma 3.4]{duchongcong24}, 
\begin{equation}\label{Nhatest}
\hat{N}(\frac{1}{2},\bx,[0,T]) \leq 2N(\frac{1}{2},\bx,[0,T]) \leq 2 + 2^{p+1} \ltn \bx \rtn^p_{\tp,[0,T]}. 
\end{equation}
Finally, by replacing $h_0 = y_0^* -\mu_0^* =0$, $L$ in \eqref{Lest} and $\hat{N}$ in \eqref{Nhatest} into \eqref{hsupest}, we obtain \eqref{hestdiscrete} for a generic random variable $\xi$.
\end{proof}

		\bibliographystyle{apalike}

	\end{document}